\documentclass[11pt,twoside, a4paper]{amsart}
\usepackage[foot]{amsaddr}

\usepackage{mathabx}
\usepackage[utf8]{inputenc}
\usepackage{physics}
\usepackage{amsmath}
\usepackage{amssymb}
\usepackage{amsthm}
\usepackage{xcolor}
\usepackage{enumerate}
\usepackage{graphicx}
\usepackage{subfigure}

\usepackage{tikz}
\usetikzlibrary{shapes,decorations,backgrounds,arrows}
\usepackage[osf,sc]{mathpazo}
\usepackage[hidelinks]{hyperref}
\usepackage[linesnumbered,ruled]{algorithm2e}
\usepackage[authoryear]{natbib}

\usepackage{hyperref}
\usepackage{hypcap}

\usepackage[capitalize]{cleveref}

\newtheorem{theorem}{Theorem}
\newtheorem{proposition}{Proposition}
\newtheorem{lemma}{Lemma}
\newtheorem{corollary}{Corollary}

\theoremstyle{definition}
\newtheorem{dfn}{Definition}
\newtheorem{assump}{Assumption}
\newtheorem{rem}{Remark}
\newtheorem{examp}{Example}

\DeclareMathOperator{\E}{\mathbb E}

\newcommand{\given}{\;\middle|\;}

\DeclareMathOperator{\Cov}{Cov}
\DeclareMathOperator{\Var}{Var}

\DeclareMathOperator*{\argmin}{arg\,min}

\renewcommand{\phi}{\varphi}
\renewcommand{\S}{\mathcal{S}}
\newcommand{\s}{\mathrm{sub}}
\newcommand{\sep}{\mathrm{sep}}
\newcommand{\bx}{\mathbf{x}}
\newcommand{\by}{\mathbf{y}}

\newcommand{\bX}{\mathbf{X}}

\newcommand{\bZ}{\mathbf{Z}}

\newcommand{\ex}{\mathbb{E}}

\renewcommand{\real}{\mathbb{R}}
\newcommand{\I}{\mathcal{I}}
\newcommand{\nat}{\mathbb{N}}
\renewcommand{\P}{\mathbb P}
\renewcommand{\epsilon}{\varepsilon}
\newcommand{\bbeta}{\boldsymbol{\beta}}

\title[Substitute Adjustment]{Substitute adjustment via recovery of 
latent variables}
\date{\today}

\author[J. Adams]{Jeffrey Adams$^\sharp$}
\email{ja@math.ku.dk}

\author[N. R. Hansen]{Niels Richard Hansen$^\sharp$}
\email{niels.r.hansen@math.ku.dk}

\address{$^\sharp$Department of Mathematical Sciences, University of Copenhagen \newline
Universitetsparken 5, Copenhagen, 2100, Denmark}

\begin{document}

\maketitle

\begin{abstract} The deconfounder was proposed as a method for estimating 
causal parameters in a context with multiple causes and unobserved 
confounding. It is based on recovery of a latent variable from the observed 
causes. We disentangle the causal interpretation from the statistical 
estimation problem and show that the deconfounder in general estimates
adjusted regression target parameters. It does so by outcome regression adjusted 
for the recovered latent variable termed the substitute. We refer to 
the general algorithm, stripped of causal assumptions, as substitute adjustment. 
We give theoretical results to support that substitute adjustment estimates 
adjusted regression parameters when the regressors are conditionally 
independent given the latent variable. We also introduce a variant of 
our substitute adjustment algorithm that estimates an assumption-lean 
target parameter with minimal model assumptions. We then give finite 
sample bounds and a\-symp\-totic results supporting substitute adjustment
estimation in the case where the latent variable takes values in a finite set. 
A simulation study illustrates finite sample properties of substitute adjustment.
Our results support that when the latent variable model of the regressors hold, 
substitute adjustment is a viable method for adjusted regression. 
\end{abstract}

\section{Introduction}

The deconfounder was proposed by \citet{wang2019blessings} as a general algorithm for estimating 
causal parameters via outcome regression when: (1) there are multiple observed 
causes of the outcome; (2) the causal effects are potentially confounded by a latent variable; 
(3) the causes are conditionally independent given a latent variable $Z$. 
The proposal spurred discussion and criticism; see the comments to \citep{wang2019blessings}
and the contributions by \citet{damour2019multi-cause, ogburn2020counterexamples} and \citet{Grimmer:2023}.
One question raised was whether the assumptions made by \citet{wang2019blessings}
are sufficient to claim that the deconfounder estimates a causal 
parameter. Though an amendment by \citet{wang2020towards} addressed the 
criticism and clarified their assumptions, it did not resolve all
questions regarding the deconfounder.

The key idea of the deconfounder is to recover the latent variable $Z$ from the 
observed causes and use this \emph{substitute confounder} as a replacement 
for the unobserved confounder. The causal parameter is then estimated by 
outcome regression using the substitute confounder for adjustment.
This way of adjusting for potential confounding 
has been in widespread use for some time in genetics and 
genomics, where, e.g.,  EIGENSTRAT based on PCA \citep{Patterson:2006, Price:2006} 
was proposed to adjust for population structure in genome wide 
association studies (GWASs); see also \citep{Song:2015}. Similarly,
surrogate variable adjustment \citep{Leek:2007} 
adjusts for unobserved factors causing unwanted variation in gene 
expression measurements. 

In our view, the discussion regarding the deconfounder 
was muddled by several issues. First, issues with non-identifiablity of target
parameters from the observational distribution with a 
\emph{finite} number of observed causes lead to confusion. 
Second, the causal role of the latent variable $Z$ and its causal relations to any unobserved confounder were difficult to grasp. Third, there was a lack of theory supporting that the deconfounder was actually 
estimating causal target parameters consistently. We defer the treatment of
the thorny causal interpretation of the deconfounder to the discussion in 
Section \ref{sec:discussion} and focus here on the statistical aspects. 

In our view, the statistical problem is best treated as 
\emph{adjusted regression} without insisting on a causal interpretation. 
Suppose that we observe a real valued outcome variable $Y$ and additional 
variables $X_1, X_2, \ldots, X_p$. We can then be interested in estimating 
the adjusted regression function 
\begin{equation} \label{eq:target-regression}
x \mapsto \ex\left[ \ex\left[Y \mid X_i = x; \bX_{-i} \right]\right]    
\end{equation}
where $\bX_{-i}$ denotes all variables but $X_i$. That is, we adjust 
for all other variables when regressing $Y$ on $X_i$. The adjusted 
regression function could have a causal interpretation in some 
contexts, but it is also of interest without a causal interpretation. 
It can, for instance, be used to study the added predictive value of 
$X_i$, and it is constant (as a function of $x$) if and only if 
$\ex\left[Y \mid X_i = x; \bX_{-i} \right] = 
\ex\left[Y \mid \bX_{-i} \right]$; 
that is, if and only if $Y$ is conditionally mean independent 
of $X_i$ given $\bX_{-i}$ \citep{lundborg2023projected}.

In the context of a GWAS, $Y$ is a continuous phenotype and $X_i$
represents a single nucleotide polymorphism (SNP)
at the genomic site $i$. The
regression function \eqref{eq:target-regression} quantifies 
how much a SNP at site $i$ adds to the prediction of the phenotype outcome 
on top of all other SNP sites. In practice, only a fraction of all SNPs
along the genome are observed, yet the number of SNPs can be 
in the millions, and estimation of the full regression model 
$\ex\left[Y \mid X_i = x; \bX_{-i} = \bx_{-i} \right]$ can be impossible
without model assumptions. Thus if the regression function  
\eqref{eq:target-regression} is the target of interest, it is 
extremely useful if we, by adjusting for a substitute of a latent variable,
can obtain a computationally efficient and statistically valid 
estimator of \eqref{eq:target-regression}.

From our perspective, when viewing the problem as that of adjusted regression, 
the most pertinent questions are: (1) when is adjustment by the latent 
variable $Z$ instead of $\bX_{-i}$ appropriate; (2) can adjustment by substitutes of the 
latent variable, recovered from the observe $X_i$-s, be justified; 
(3) can we establish an asymptotic theory that allows for 
statistical inference when adjusting for substitutes?

With the aim of answering the three questions above, this paper 
makes two main contributions:

{\bf A transparent statistical framework}. We focus on estimation of the adjusted mean, 
thereby disentangling the statistical problem from the causal discussion. This way 
the target of inference is clear and so are the assumptions we need about the observational 
distribution in terms of the latent variable model. We present in Section \ref{sec:setup} a general framework with an infinite number of $X_i$-variables, and we 
present clear assumptions implying that we can 
replace adjustment by $\bX_{-i}$ with adjustment by $Z$. 
Within the general framework, we subsequently present an 
assumption-lean target parameter that is interpretable without restrictive 
model assumptions on the regression function. 

{\bf A novel theoretical analysis}. By restricting attention to 
the case where the latent variable $Z$ takes values in a finite set,
we give in Section \ref{sec:finite_mixture_model_full} bounds on the estimation 
error due to using substitutes and on the recovery 
error---that is, the substitute mislabeling rate. These bounds quantify,
among other things, how the errors depend on $p$; the actual 
(finite) number of $X_i$-s used for recovery. 
With minimal assumptions on the conditional distributions in the latent 
variable model and on the outcome model, we use our bounds to derive 
asymptotic conditions 
ensuring that the assumption-lean target parameter can be estimated 
just as well using substitutes as if the latent variables were observed.

To implement substitute adjustment in practice, we leverage recent developments on 
estimation in finite mixture models via tensor methods, which are computationally 
and statistically efficient in high dimensions. We illustrate our results 
via a simulation study in Section \ref{sec:simulations}. Proofs and 
auxiliary results are in Appendix \ref{app:proofs}. Appendix \ref{app:kakutani}
contains a complete characterization of when recovery of $Z$ is possible 
from an infinite $\bX$ in a Gaussian mixture model.

\subsection{Relation to existing literature}

Our framework and results are based on ideas by \citet{wang2019blessings, wang2020towards}
and the literature preceding them on adjustment by surrogate/substitute variables. 
We add new results to this line of research on the theoretical justification 
of substitute adjustment as a method for estimation. 

There is some literature on the theoretical properties of tests and estimators 
in high-dimensional problems with latent variables. Somewhat related to our 
framework is the work by \cite{wang2017} on adjustment for latent confounders 
in multiple testing, motivated by applications to gene expression analysis.   
More directly related is the work by \cite{Cevid:2020} and \cite{Guo:2022}, 
who analyze estimators within a linear modelling framework with unobserved 
confounding. While their methods and results are definitely interesting, they 
differ from substitute adjustment, since they do 
not directly attempt to recover the latent variables. 
The linearity and sparsity assumptions, which we will not make, 
play an important role for their methods and analysis. 

The paper by \citet{Grimmer:2023} comes closest to our framework and analysis. 
\citet{Grimmer:2023} present theoretical results and extensive numerical 
examples, primarily with a continuous latent variable. Their results are 
not favorable for the deconfounder and they 
conclude that the deconfounder is ``not a viable substitute for careful research design in real-world applications''. Their theoretical analyses 
are mostly in terms of computing the population (or $n$-asymptotic) bias of a method 
for a finite $p$ (the number of $X_i$-variables), and then possibly investigate 
the limit of the bias as $p$ tends to infinity. Compared to this, we 
analyze the asymptotic behaviour of the estimator based on substitute 
adjustment as $n$ and $p$ tend to infinity jointly.
Moreover, since we specifically treat discrete latent variables, some of our 
results are also in a different framework.

\section{Substitute adjustment} \label{sec:setup}
\subsection{The General Model}\label{sec:model_general_form}

The full model is specified in terms of variables $(\bX, Y)$, where 
$Y\in \real$ is a real valued outcome variable of interest and 
$\bX \in \real^{\nat}$ is a infinite vector of additional real valued variables. 
That is, $\bX = (X_i)_{i \in \nat}$ with $X_i \in \real$ for $i \in \nat$.
We let $\bX_{-i} = (X_j)_{j \in \nat \backslash{\{i\}}}$, and define (informally) 
for each $i \in \nat$ and $x \in \real$ the target parameter of interest
\begin{equation} \label{eq:target0}
    \chi_x^i = \ex\left[\ex\left[Y \given X_i = x; \bX_{-i}\right] \right].
\end{equation}
That is, $\chi_x^i$ is the mean outcome given $X_i = x$ when adjusting for 
all remaining variables $\bX_{-i}$. Since $\ex\left[Y \mid X_i = x; \bX_{-i}\right]$
is generally not uniquely defined for all $x \in \real$ by the distribution of 
$(\bX, Y)$, we need some additional structure to formally define $\chi^i_x$. 
The following assumption and subsequent definition achieve 
this by assuming that a particular choice of the conditional expectation is
made and remains fixed. Throughout,
$\real$ is equipped with the Borel $\sigma$-algebra and $\real^\nat$ with
the corresponding product $\sigma$-algebra. 

\begin{assump}[Regular Conditional Distribution] \label{ass:regpos} \rm
Fix for each 
$i \in \nat$ a Markov kernel
$(P_{x, \bx}^i)_{(x, \bx) \in \real \times \real^\nat}$ on $\real$.
Assume that $P_{x, \bx}^i$ is the regular conditional distribution 
of $Y$ given $(X_i, \bX_{-i}) = (x, \bx)$ for all $x \in \real$, $\bx \in \real^\nat$ 
and $i \in \nat$. With $P^{-i}$ the distribution of $\bX_{-i}$, 
suppose additionally that 
\[
    \iint |y| \, P^i_{x, \bx} (\mathrm{d} y) P^{-i}(\mathrm{d} \bx) < \infty
\]
for all $x \in \real$.
\end{assump}



\begin{dfn} \label{dfn:chix} 
Under Assumption \ref{ass:regpos} we define  
\begin{equation} \label{eq:target}
    \chi_x^i = \iint y \, P^i_{x, \bx} (\mathrm{d} y)
    P^{-i} (\mathrm{d}\bx).
\end{equation}
\end{dfn}

\begin{rem} \rm 
Definition \ref{dfn:chix} makes the choice of conditional expectation explicit
by letting
\[\ex\left[Y \mid X_i = x; \bX_{-i}\right] = \int y \, P^i_{x, \bX_{-i}} (\mathrm{d} y)\]
be defined in terms of the specific regular conditional distribution 
that is fixed according to Assumption \ref{ass:regpos}. We may need additional 
regularity assumptions to identify this Markov kernel from the distribution of $(\bX, Y)$, which we will not pursue here.
\end{rem}

The main assumption in this paper is the existence of a latent variable, $Z$, that will 
render the $X_i$-s conditionally independent, and which can be recovered 
from $\bX$ in a suitable way. The variable $Z$ will take values in a measurable 
space $(E, \mathcal{E})$, which we assume to be a Borel space. We use the 
notation $\sigma(Z)$ and $\sigma(\bX_{-i})$ to denote the $\sigma$-algebras 
generated by $Z$ and $\bX_{-i}$, respectively. 

\begin{assump}[Latent Variable Model] \label{ass:variable} \rm 
There is a random variable 
$Z$ with values in $(E, \mathcal{E})$ such that: 
\begin{enumerate}
\item $X_1, X_2, \ldots$ are conditionally independent given $Z$,
\item $\sigma(Z) \subseteq \bigcap_{i=1}^{\infty} \sigma(\bX_{-i})$.
\end{enumerate}
\end{assump}

The latent variable model given by Assumption \ref{ass:variable} allows us to 
identify the adjusted mean by adjusting for the latent variable only.  


\begin{proposition} \label{prop:variablerep} Fix $i \in \nat$ and let 
$P^{-i}_{z}$ denote a regular conditional distribution of $\bX_{-i}$ given $Z = z$. 
Under Assumptions \ref{ass:regpos} and \ref{ass:variable}, the Markov kernel 
\begin{equation} \label{eq:Qkernel}
Q_{x,z}^i(A) = \int P^i_{x, \bx}(A) P^{-i}_z(\mathrm{d} \bx), \qquad A \subseteq \real
\end{equation}
is a regular conditional distribution of $Y$ given $(X_i,Z) = (x,z)$, in which case
\begin{equation} \label{eq:target2}
    \chi_x^i = \iint y \, Q_{x,z}^i (\mathrm{d} y) P^Z (\mathrm{d}z) = \ex\left[ \ex\left[ Y \mid X_i = x; Z \right] \right].
\end{equation}
\end{proposition}

\begin{figure}
    \centering
\begin{tikzpicture}[scale=1]
\tikzset{vertex/.style = {shape=circle,draw,minimum size=2em}}
\tikzset{edge/.style = {->,> = latex', thick}}

\node[vertex] (Z) at  (-2,0) {$Z$};
\node[vertex] (X1) at  (0,-1) {$X_i$};
\node[vertex] (bX) at  (0,1) {$\bX_{-i}$};
\node[vertex] (Y) at  (2,0) {$Y$};


\draw[edge] (Z) to (X1);
\draw[edge] (Z) to (bX);
\draw[edge] (X1) to (Y);
\draw[edge] (bX) to (Y);

\end{tikzpicture}
    \caption{Directed Acyclic Graph (DAG) representing the joint distribution of $(X_i, \bX_{-i}, Z, Y)$. The variable $Z$ blocks the backdoor from $X_i$ to $Y$. \label{fig:backdoor}}
\end{figure}
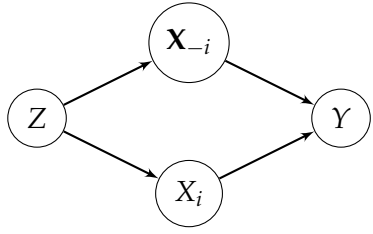

The joint distribution of $(X_i, \bX_{-i}, Z, Y)$ is, by Assumption \ref{ass:variable}, 
Markov w.r.t. to the graph in Figure \ref{fig:backdoor}. Proposition \ref{prop:variablerep}
is essentially the backdoor criterion, since $Z$ blocks the backdoor from $X_i$ to 
$Y$ via $\bX_{-i}$; see Theorem 3.3.2 in \citep{pearl2009causality} or Proposition 6.41(ii) in \citep{Peters2017}. Nevertheless, we include a proof in Appendix \ref{app:proofs} for two reasons. 
First, Proposition \ref{prop:variablerep} does not involve  
causal assumptions about the model, and we want to clarify that the mathematical result is 
agnostic to such assumptions. Second, the proof we give of Proposition \ref{prop:variablerep} does not require regularity assumptions, such as densities of the conditional 
distributions, but it relies subtly on Assumption \ref{ass:variable}(2).

\begin{examp} \label{ex:lin} \rm
Suppose $\ex[|X_i|] \leq C$  for all $i$ and 
some finite constant $C$, and assume, for simplicity, that $\ex[X_i] = 0$.
Let $\bbeta = (\beta_i)_{i \in \nat} \in \ell_1$ and define 
\[ \langle \bbeta, \bX \rangle = \sum_{i=1}^{\infty} \beta_i X_i.\]
The infinite sum converges almost surely since $\bbeta \in \ell_1$.
With $\epsilon$ being $\mathcal{N}(0,1)$-distributed and independent of $\bX$ 
consider the outcome model 
\[Y = \langle \bbeta, \bX \rangle  + \epsilon.\]
Letting $\bbeta_{-i}$ denote the $\bbeta$-sequence with the $i$-th coordinate removed, 
a straightforward, though slightly informal, computation, gives 
\begin{align*}
    \chi^i_x & = \ex\left[ \ex\left[ \beta_i X_i + \langle \bbeta_{-i}, \bX_{-i} \rangle
\mid X_i = x; \bX_{-i} \right]\right] \\
& = \beta_i x + \ex\left[ \langle \bbeta_{-i}, \bX_{-i} \rangle \right] 
= \beta_i x + \langle \bbeta_{-i}, \ex\left[ \bX_{-i} \right] \rangle
= \beta_i x.
\end{align*}

To fully justify the computation, via Assumption \ref{ass:regpos}, we let 
$P^{i}_{x, \bx}$
be the $\mathcal{N}(\beta_i x + \langle \bbeta_{-i}, \bx \rangle, 1)$-distribution for the 
$P^{-i}$-almost all $\bx$ where $\langle \bbeta_{-i}, \bx \rangle$ is well defined. 
For the remaining $\bx$ we let $P^{i}_{x, \bx}$ be the $\mathcal{N}(\beta_i x, 1)$-distribution.
Then $P^{i}_{x, \bx}$ is a regular conditional distribution of $Y$ given $(X_i,\bX_{-i}) = (x,\bx)$, 
\[\int y \, P_{x, \bx}^i (\mathrm{d} y) = \beta_i x + \langle \bbeta_{-i}, \bx \rangle \quad \text{for } P^{-i}\text{-almost all } \bx,\]
and $\chi^i_x = \beta_i x$ follows from \eqref{eq:target}. It also follows from 
\eqref{eq:Qkernel} that for $P^Z$-almost all $z \in E$,
\begin{align*}
    \ex\left[Y \mid X_i = x; Z = z\right] & = \int y \, Q^i_{x,z}(\mathrm{d} y) \\
    & = \beta_i x + \int \langle \bbeta_{-i}, \bx \rangle P^{-i}_z(\mathrm{d}\bx) \\
    & = \beta_i x + \sum_{j \neq i} \beta_j \ex[X_j \mid Z = z].
\end{align*}
That is, with
$\Gamma_{-i}(z) = \sum_{j \neq i} \beta_j \ex[X_j \mid Z = z]$, the regression model 
\[
    \ex\left[Y \mid X_i = x; Z = z\right] = \beta_i x + \Gamma_{-i}(z)
\] 
is a partially linear model. 
\end{examp}

\begin{examp} \rm 
While Example \ref{ex:lin} is explicit about the outcome model, it 
does not describe an explicit latent variable model fulfilling 
Assumption \ref{ass:variable}. To this end, take $E = \real$, let $Z', U_1, U_2, \ldots$ be 
i.i.d. $\mathcal{N}(0,1)$-distributed and set $X_i = Z' + U_i$. By the 
Law of Large Numbers, for any $i \in \nat$,
\[\frac{1}{n} \sum_{j = 1; j \neq i}^{n+1} X_j = Z' + \frac{1}{n} \sum_{j = 1; j \neq i}^{n+1} U_j
\rightarrow Z'\]
almost surely for $n \to \infty$. Setting 
\[ Z = \left\{ \begin{array}{ll} 
\lim\limits_{n \to \infty} \frac{1}{n} \sum_{j = 1; j \neq i}^{n+1} X_j & \quad \text{if the limit exists} \\ 0 & \quad \text{otherwise} \end{array} \right.
\]
we get that $\sigma(Z) \subseteq \sigma(\bX_{-i})$ for any $i \in \nat$ and 
$Z = Z'$ almost surely. Thus, Assumption \ref{ass:variable} holds.

Continuing with the outcome model from Example \ref{ex:lin}, we see that 
for $P^Z$-almost all $z \in E$,
\[ \ex[X_j \mid Z = z] = \ex[Z' + U_j \mid Z = z] = z,\]
thus $\Gamma_{-i}(z) = \gamma_{-i} z$ with $\gamma_{-i} = \sum_{j \neq i} \beta_j$.
In this example it is actually possible to compute the regular conditional 
distribution, $Q^i_{x,z}$, of $Y$ given $(X_i, Z) = (x, z)$ explicitly. It is the  
$\mathcal{N}\left(\beta_i x + \gamma_{-i} z, 1 + \| \bbeta_{-i} \|^2_2 \right)$-distribution
where \mbox{$\| \bbeta_{-i} \|^2_2 = \langle \bbeta_{-i}, \bbeta_{-i} \rangle$}.
\end{examp}

\subsection{Substitute Latent Variable Adjustment}
\label{sec:deconf_general_form}

Proposition \ref{prop:variablerep} tells us that under 
Assumptions 1 and 2 the adjusted mean, $\chi_x^i$, defined by adjusting for 
the entire infinite vector $\bX_{-i}$, is also given by adjusting 
for the latent variable $Z$. If the latent variable were observed 
we could estimate $\chi^i_x$ in terms of 
an estimate of the following regression function. 

\begin{dfn}[Regression function] \label{dfn:b} Under Assumptions 1 and 2 define the 
regression function 
\begin{equation}
    b_x^i(z) = \int y \, Q_{x,z}^i(\mathrm{d}y) = \ex\left[Y \mid X_i = x; Z = z\right]
\end{equation}
where $Q_{x,z}^i$ is given by \eqref{eq:Qkernel}.
\end{dfn}

If we had $n$ i.i.d. observations, 
$(x_{i,1}, z_1, y_1), \ldots, (x_{i,n}, z_n, y_n)$, of $(X_i, Z, Y)$,
a straightforward plug-in estimate of $\chi^i_x$ is  
\begin{equation}
    \hat{\chi}^i_x = \frac{1}{n} \sum_{k=1}^n \hat{b}_x^i(z_k),
\end{equation}
where $\hat{b}_x^i(z)$ is an estimate of the regression function 
$b_x^i(z)$. 
In practice we do not observe the latent variable $Z$. Though Assumption \ref{ass:variable}(2) 
implies that $Z$ can be recovered from $\bX$, we do not assume we know 
this recovery map, nor do we in practice observe the entire $\bX$, but only the 
first $p$ coordinates, $\bX_{1:p} = (X_1, \ldots, X_p)$. 

We thus need an 
estimate of a recovery map, $\hat{f}^p : \real^p \to E$, such 
that for the \emph{substitute latent variable}
$\hat{Z} = \hat{f}^p(\bX_{1:p})$ we have\footnote{We can in general only hope 
to learn a recovery map of $Z$ up to a Borel isomorphism, but this is 
also all that is needed, cf. Assumption \ref{ass:variable}.} that $\sigma(\hat{Z})$ 
approximately contains the same information as $\sigma(Z)$.  
Using such substitutes, a natural way to estimate $\chi^i_x$ is given by Algorithm \ref{alg:generalalg}, 
which is a general three-step procedure returning the estimate $\widehat{\chi}^{i,\s}_x$.

\begin{algorithm}[t]
    \caption{General Substitute Adjustment} \label{alg:generalalg}
    \textbf{input:} data $\S_0 = \{\bx_{1:p,1}^0, \ldots, \bx_{1:p,m}^0\}$ and 
    $\S = \{(\bx_{1:p,1},y_1), \ldots (\bx_{1:p,n}, y_{n})\}$, a 
    set $E$, $i \in \{1, \ldots, p\}$ and $x \in \real$\;
    \textbf{options:} a method for estimating a recovery map $f^p : \real^p \to E$, 
    a method for estimating the regression function $z \mapsto b_x^i(z)$\;
    \Begin{
    use data in $\S_0$ to compute the estimate $\hat f^p$
    of the recovery map.
    
    use data in $\S$ to
    compute the substitute latent variables as $\hat{z}_k := \hat f^p(\bx_{1:p,k})$, 
    $k = 1, \ldots, n$.

    use data in $\S$ combined with the substitutes to compute the regression 
    function estimate, $z \mapsto \hat{b}_x^i(z)$, and set
    \begin{equation} \nonumber
        \widehat{\chi}^{i,\s}_x = \frac{1}{n} \sum_{k=1}^{n} 
        \hat{b}_x^i(\hat{z}_k).
    \end{equation}
    }
    \Return{$\widehat{\chi}^{i,\s}_x$}
\end{algorithm}

The regression estimate $\hat{b}_x^i(z)$ in Algorithm \ref{alg:generalalg} 
is computed on the basis of the substitutes, which likewise enter into 
the final computation of $\widehat{\chi}^{i,\s}_x$. Thus the estimate 
is directly estimating $\chi^{i,\s}_x = \ex\left[ \ex\left[Y \mid X_i = x; \hat{Z}\right] \given \hat{f}^p \right]$, and it is expected to be biased as an estimate of 
$\chi^i_x$. The general idea is that under 
some regularity assumptions, and for $p \to \infty$ and $m \to \infty$ appropriately,
$\chi^{i,\s}_x \to \chi^i_x$ and the bias vanishes asymptotically. 
Section \ref{sec:finite_mixture_model_full} specifies a setup where such a 
result is shown rigorously. 

Note that the estimated recovery map $\hat{f}^p$ in Algorithm \ref{alg:generalalg} 
is the same for all $i = 1, \ldots, p$. Thus for any fixed $i$, the $x_{i,k}^0$-s are 
used for estimation of the recovery map, and the $x_{i,k}$-s are used 
for the computation of the substitutes. Steps 4 and 5 of the algorithm
could be changed to construct a recovery map $\hat{f}_{-i}^p$ 
independent of the $i$-th coordinate. This appears to align better with 
Assumption \ref{ass:variable}, and it would most likely make the $\hat{z}_k$-s 
slightly less correlated with the $x_{i,k}$-s. It would, on the other hand, lead 
to a slightly larger recovery error, and worse, a substantial increase in 
the computational complexity if we want to estimate $\widehat{\chi}_x^{i,\s}$
for all $i = 1, \ldots, p$.

Algorithm \ref{alg:generalalg} leaves some options open. First, the estimation method 
used to compute $\hat f^p$ could be based on any method for estimating a
recovery map, e.g., using a factor model if $E = \real$ or a mixture model if $E$ is finite. 
The idea of such methods is to compute a parsimonious $\hat f^p$ such that: (1) conditionally 
on $\hat{z}^0_k = \hat f^p(\bx_{1:p,k}^0)$ the observations $x_{1,k}^0, \ldots, x_{p,k}^0$ 
are approximately independent for $k = 1, \ldots, m$; and (2) $\hat{z}^0_k$
is minimally predictive of $x_{i,k}^0$ for $i = 1, \ldots, p$.
Second, the regression method for estimation of the regression 
function $b_x^i(z)$ could be any parametric or nonparametric 
method. If $E = \real$ we could use OLS combined with the parametric model 
$b_x^i(z) = \beta_0 + \beta_i x + \gamma_{-i} z$, which would lead to 
the estimate 
\[\widehat{\chi}^{i,\s}_x = \hat{\beta}_0 + \hat{\beta}_i x + 
\hat{\gamma}_{-i} \frac{1}{n} \sum_{k=1}^{n} \hat{z}_k.\]
If $E$ is finite, we could still use OLS but now combined with the parametric model 
$b_x^i(z) = \beta_{i,z}' x + \gamma_{-i,z}$, 
which would lead to the estimate 
\[\widehat{\chi}^{i,\s}_x =  \left( \frac{1}{n} \sum_{k=1}^{n} 
\hat{\beta}_{i,\hat{z}_k}' \right) x + \frac{1}{n} \sum_{k=1}^{n} 
\hat{\gamma}_{-i, \hat{z}_k}.\]

The relation between the two datasets in Algorithm \ref{alg:generalalg} is not 
specified by the algorithm either. It is possible that they are 
independent, e.g., by data splitting, in which case $\hat{f}^p$ is 
independent of the data in $\S$. It is also possible that $m = n$ and 
$\bx_{1:p,k}^0 = \bx_{1:p,k}$ for $k = 1, \ldots, n$. While we will assume 
$\S_0$ and $\S$ independent for the theoretical analysis, the $\bx_{1:p}$-s 
from $\S$ will  in practice often be part of $\S_0$, if not all of $\S_0$.

\subsection{Assumption-Lean Substitute Adjustment}
If the regression model in the general Algorithm \ref{alg:generalalg} 
is misspecified we cannot expect that 
$\widehat{\chi}^{i,\s}_x$ is a consistent estimate of 
$\chi^i_x$. In Section \ref{sec:finite_mixture_model_full} we investigate 
the distribution of a substitute adjustment estimator in the case where $E$
is finite. It is possible to carry out this investigation assuming a partially 
linear regression model, $b_x^i(z) = \beta_{i} x + \Gamma_{-i}(z)$, but the 
results would then hinge on this model being correct. To circumvent 
such a model assumption we proceed instead in the spirit of \emph{assumption-lean 
regression} \citep{Berk:2021, Vansteelandt:2022}. Thus we focus on a 
univariate target parameter defined as a functional of the data distribution,
and we then investigate its estimation via substitute adjustment. 

\begin{assump}[Moments] \label{ass:moment} \rm It holds that 
$\ex(Y^2) < \infty$, \mbox{$\ex[X_i^2] < \infty$} and 
$\ex\left[\Var \left[X_i \mid Z \right] \right] > 0$.
\end{assump}

\begin{dfn}[Target parameter] \label{dfn:beta-i}
Let $i \in \nat$. Under Assumptions \ref{ass:variable} and \ref{ass:moment}
define the target parameter
\begin{equation} \label{eq:target-beta}
    \beta_i = \frac{\E\left[\Cov \left[X_i, Y \mid Z \right] \right]}
{\E\left[\Var \left[X_i \mid Z \right] \right]}.
\end{equation}
\end{dfn}

\begin{algorithm}
    \caption{Assumption-Lean Substitute Adjustment} \label{alg:leanalg}
    \textbf{input:} data $\S_0 = \{\bx_{1:p,1}^0, \ldots, \bx_{1:p,m}^0\}$ and 
    $\S = \{(\bx_{1:p,1},y_1), \ldots (\bx_{1:p,n}, y_{n})\}$, a set $E$ and 
    $i \in \{1, \ldots, p\}$\;
    \textbf{options:} a method for estimating the recovery map $f^p : \real^p \to E$, 
    methods for estimating the regression functions $\mu_i(z) = \ex\left[X_i \mid Z = z\right]$
    and $g(z) = \ex\left[Y \mid Z = z\right]$\;
    \Begin{
    use data in $\S_0$ to compute the estimate $\hat f^p$ of the recovery map. 
    
    use data in $\S$ to 
    compute the substitute latent variables as $\hat{z}_k := \hat f^p(\bx_{1:p,k})$, 
    $k = 1, \ldots, n$.

    use data in $\S$ combined with the substitutes to compute the regression function estimates $z \mapsto \hat{\mu}_i(z)$ and $z \mapsto \hat{g}(z)$, and set
    \begin{equation} \nonumber
        \widehat{\beta}^{\s}_i = \frac{
        \sum_{k=1}^{n} (x_{i,k}- \hat{\mu}_i(\hat{z}_k))(y_k - \hat{g}(\hat{z}_k))
    }{
        \sum_{k=1}^{n} (x_{i,k}- \hat{\mu}_i(\hat{z}_k))^2
    }.
    \end{equation}
    }
    \Return{$\widehat{\beta}^{\s}_i$}
\end{algorithm}

Algorithm \ref{alg:leanalg} gives a procedure for estimating 
$\beta_i$ based on substitute latent variables. The following 
proposition gives insight on the interpretation of the 
target parameter $\beta_i$.

\begin{proposition} \label{prop:lean-rep}
Under Assumptions 1, 2 and 3, and with $b_x^i(z)$ given as in Definition \ref{dfn:b}, 
and $\beta_i$ given as in Definition \ref{dfn:beta-i},
\begin{equation} \label{eq:beta-i-rep}
    \beta_i = \frac{\E\left[\Cov \left[X_i, b^i_{X_i}(Z) \mid Z \right] \right]}
{\E\left[\Var \left[X_i \mid Z \right] \right]}.
\end{equation}
Moreover, $\beta_i = 0$ if $b^i_{x}(z)$ does not depend on $x$. If 
$b^i_{x}(z) = \beta_i'(z) x + \Gamma_{-i}(z)$
then 
\begin{equation} \label{eq:beta-i-rep-lin}
    \beta_i = \ex\left[w_i(Z) \beta_i'(Z)\right]
\end{equation}
where 
\[w_i(Z) = \frac{\Var[X_i \mid Z]}{\ex\left[\Var[X_i \mid Z] \right]}.\]
\end{proposition}

We include a proof of Proposition \ref{prop:lean-rep} in Appendix \ref{app:model_proofs} 
for completeness. The arguments are essentially as in \citep{Vansteelandt:2022}.

\begin{rem} \rm
If $b^i_{x}(z) = \beta_i'(z) x + \Gamma_{-i}(z)$ it follows from Proposition \ref{prop:variablerep} that $\chi^i_x = \beta'_i x$, where the coefficient  
$\beta'_i = \ex[\beta_i'(Z)]$ may differ from $\beta_i$ given by \eqref{eq:beta-i-rep-lin}. 
In the special case where the variance of $X_i$ given $Z$ is constant across all values of $Z$, 
the weights in \eqref{eq:beta-i-rep-lin} are all $1$, in which case $\beta_i = \beta'_i$.
For the partially linear model, $b^i_{x}(z) = \beta_i' x + \Gamma_{-i}(z)$, 
with $\beta_i'$ not depending on $z$, it follows from \eqref{eq:beta-i-rep-lin}
that $\beta_i = \beta_i'$ irrespectively of the weights.
\end{rem}

\begin{rem} \rm
If $X_i \in \{0, 1\}$ then $b_x^i(z) = (b_1^i(Z) - b_0^i(Z))x + b_0^i(Z)$,
and the contrast $\chi^i_1 - \chi^i_0 = \ex\left[b_1^i(Z) - b_0^i(Z)\right]$ is an 
unweighted mean of differences, while it follows from \eqref{eq:beta-i-rep-lin} that 
\begin{equation} \label{eq:beta-i-rep-dic}
    \beta_i = \ex\left[w_i(Z) (b_1^i(Z) - b_0^i(Z))\right].
\end{equation}
If we let $\pi_i(Z) = \P(X_i = 1 \mid Z)$, we see that the weights are given as 
\[w_i(Z) = \frac{\pi_i(Z)(1-\pi_i(Z))}{\ex\left[\pi_i(Z)(1-\pi_i(Z))\right]}.\]
\end{rem}

We summarize three important take-away messages from Proposition \ref{prop:lean-rep} and the remarks above as follows:

\begin{description}
    \item[Conditional mean independence] 
    The null hypothesis of conditional mean independence, 
    \[\ex\left[Y \mid X_i = x; \mathbf{X}_{-i}\right]) = 
    \ex\left[Y \mid \mathbf{X}_{-i}\right],\] implies that 
    $\beta_i = 0$. The target parameter $\beta_i$ thus suggests an assumption-lean 
    approach to testing this null without a specific model of the 
    conditional mean.
    \item[Heterogeneous partial linear model] 
    If the conditional mean, 
    \[b_x^i(z) = \ex\left[Y \mid X_i = x; Z = z\right],\]
    is linear in $x$ with an $x$-coefficient that 
    depends on $Z$ (heterogeneity), the target parameter $\beta_i$ is a \emph{weighted} 
    mean of these coefficients, while $\chi^i_x = \beta'_i x$ with $\beta'_i$ the 
    \emph{unweighted} mean. 
    \item[Simple partial linear model] 
    If the conditional mean is linear in $x$ with an 
    $x$-coef{-}ficient that is \emph{independent} of $Z$ (homogeneity), 
    the target parameter $\beta_i$ coincides with this $x$-coefficient 
    and $\chi_x^i = \beta_i x$. Example \ref{ex:lin} is a special case where 
    the latent variable model is arbitrary but the full outcome model is linear.
\end{description}

Just as for the general Algorithm \ref{alg:generalalg}, the estimate that 
Algorithm \ref{alg:leanalg} outputs, $\widehat{\beta}^{\s}_i$, is not 
directly estimating the target parameter $\beta_i$. It is directly
estimating  
\begin{equation}
    \beta^{\s}_i =\frac{\E\left[\Cov \left[X_i, Y \mid \hat{Z} \right] \given \hat{f}^p \right]}
{\E\left[\Var \left[X_i \mid \hat{Z} \right] \given \hat{f}^p \right]}.
\end{equation}
Fixing the estimated recovery map $\hat{f}^p$ and letting $n \to \infty$, we 
can expect that $\widehat{\beta}^{\s}_i$ is consistent for $\beta^{\s}_i$
and not for $\beta_i$.

Pretending that the $z_k$-s were observed, we introduce the oracle estimator
\begin{equation} \nonumber
        \widehat{\beta}_i = \frac{
        \sum_{k=1}^{n} (x_{i,k}- \overline{\mu}_i(z_k))(y_k - \overline{g}(z_k))
    }{
        \sum_{k=1}^{n} (x_{i,k}- \overline{\mu}_i(z_k))^2
    }.
\end{equation}
Here, $\overline{\mu}_i$ and $\overline{g}$ denote estimates of the regression 
functions $\mu_i$ and $g$, respectively, using the $z_k$-s instead of the substitutes. 
The estimator $\widehat{\beta}_i$ is independent of $m$, $p$, and $\hat{f}^p$, and 
when $(x_{i,1}, z_1, y_1), \ldots, (x_{i,n}, z_n, y_n)$ are i.i.d. observations, 
standard regularity assumptions \citep{vanderVaart:1998} will ensure that the estimator $\widehat{\beta}_i$ 
is consistent for $\beta_i$ (and possibly even $\sqrt{n}$-rate asymptotically normal).
Writing
\begin{equation}\label{eq:beta_hat_tilde_decomposition}
    \widehat{\beta}^{\s}_i - \beta_i = (\widehat{\beta}^{\s}_i - \widehat \beta_i) + (\hat\beta_i - \beta_i)
\end{equation}
we see that if we can appropriately bound the error, 
$|\widehat{\beta}^{\s}_i - \widehat \beta_i|$, due 
to using the substitutes instead of the unobserved $z_k$-s, we can transfer asymptotic 
properties of $\hat\beta_i$ to $\widehat{\beta}^{\s}_i$. It is 
the objective of the following section to demonstrate how such a bound can 
be achieved for a particular model class.

\section{Substitute adjustment in a mixture model}
\label{sec:finite_mixture_model_full}
In this section, we present a theoretical analysis of assumption-lean substitute adjustment 
in the case where the latent variable takes values in a finite set. We provide finite-sample bounds on the error of $\widehat{\beta}^{\s}_i$ 
due to the use of substitutes, and we 
show, in particular, that there exist trajectories of $m$, $n$ and $p$ along which 
the estimator is asymptotically equivalent to the oracle estimator $\widehat{\beta}_i$,  
which uses the actual latent variables. 

\subsection{The mixture model} \label{sec:mixture}
To be concrete, we assume that $\bX$ is generated by a finite mixture model such that 
conditionally on a latent variable $Z$ with values in a finite set, the coordinates of 
$\bX$ are independent. The precise model specification is as follows.

\begin{assump}[Mixture Model] \label{ass:mixture} \rm
There is a latent variable $Z$ with values in the finite set $E = \{1, \ldots, K\}$ 
such that $X_1, X_2, \ldots$ are conditionally independent given $Z = z$. Furthermore, 
\begin{enumerate}
    \item The conditional distribution of $X_i$ given $Z = z$ has finite second moment, 
    and its conditional mean and variance are denoted
    \begin{align*} 
        \mu_i(z) & = \ex[X_i \mid Z = z] \\
        \sigma_i^2(z) & = \Var [X_i \mid Z = z]
    \end{align*}
for $z \in E$ and $i \in \nat$.
    \item The conditional means satisfy the following \emph{separation} condition
    \begin{equation} \label{eq:sep}
\sum_{i=1}^{\infty} (\mu_i(z) - \mu_i(v))^2 = \infty
\end{equation}
for all $z, v \in E$ with $v \neq z$.
\item There are constants  $0 < \sigma_{\min}^2 \leq \sigma^2_{\max} < \infty$ that bound the conditional variances;
\begin{equation} \label{eq:sigmabound}
\sigma_{\min}^2 \leq \max_{z \in E} \sigma_i^2(z) \leq \sigma^2_{\max}
\end{equation}
for all $i \in \nat$.
\item $\P(Z = z) > 0$ for all $z \in E$.
\end{enumerate}
\end{assump}

Algorithm \ref{alg:leanalgmix} is one specific version of Algorithm \ref{alg:leanalg} 
for computing $\hat{\beta}_i^{\s}$ when the latent variable takes values in a finite set $E$. 
The recovery map in Step 5 is given by computing the nearest mean, and it is thus 
estimated in Step 4 by estimating the means for each of the mixture components. How 
this is done precisely is an option of the algorithm. Once the substitutes are computed,
outcome means and $x_{i,k}$-means are (re)computed within each component. The 
computations in Steps 6 and 7 of Algorithm \ref{alg:leanalgmix} result in the same 
estimator as the OLS estimator of $\beta_i$ when it is computed using the linear model 
\[
b_x^i(z) = \beta_i x + \gamma_{-i, z}, \qquad \beta_i, \gamma_{-i, 1}, \ldots, \gamma_{-i,K} \in \real 
\]
on the data $(x_{i,1}, \hat{z}_1,y_1), \ldots (x_{i,n}, \hat{z}_n, y_{n})$. This may be 
relevant in practice, but it is also used in the proof of Theorem \ref{thm:error_bound}.
The corresponding oracle estimator, $\hat{\beta}_i$, is similarly an OLS estimator.

\begin{algorithm}
    \caption{Assumption Lean Substitute Adjustment w. Mixtures} \label{alg:leanalgmix}
    \textbf{input:} data $\S_0 = \{\bx_{1:p,1}^0, \ldots, \bx_{1:p,m}^0\}$ and 
    $\S = \{(\bx_{1:p,1},y_1), \ldots (\bx_{1:p,n}, y_{n})\}$, a finite set $E$
    and $i \in \{1, \ldots, p\}$\;
    \textbf{options:} a method for estimating the conditional means 
    $\mu_j(z) = \ex[X_j \mid Z = z]$\;
    \Begin{
    use the data in $\S_0$ to compute the estimates $\check{\mu}_j(z)$
    for $j \in \{1, \ldots, p\}$ and $z \in E$. 
    
    use the data in $\S$ to
    compute the substitute latent variables as $\hat{z}_k = 
    \argmin_{z} \|\bx_{1:p, k} - \check{\boldsymbol{\mu}}_{1:p}(z)\|_2$, 
    $k = 1, \ldots, n$.

    use the data in $\S$ combined with the substitutes to compute the estimates
    \begin{align*}
        \hat{g}(z) & = \frac{1}{\hat{n}(z)} \sum_{k: \hat{z}_k = z} y_k, \quad z \in E \\
        \hat{\mu}_i(z) & = \frac{1}{\hat{n}(z)} \sum_{k: \hat{z}_k = z} x_{i,k}, \quad z \in E,
    \end{align*}
    where $\hat{n}(z) = \sum_{k=1}^n 1(\hat{z}_k = z)$ is the number of $k$-s with $\hat{z}_k = z$.

    use the data in $\S$ combined with the substitutes to compute 
    \begin{equation} \nonumber
        \widehat{\beta}^{\s}_i = \frac{
        \sum_{k=1}^{n} (x_{i,k}- \hat{\mu}_i(\hat{z}_k))(y_k - \hat{g}(\hat{z}_k))
    }{
        \sum_{k=1}^{n} (x_{i,k}- \hat{\mu}_i(\hat{z}_k))^2
    }.
    \end{equation}
    }
    \Return{$\widehat{\beta}^{\s}_i$}
\end{algorithm}

Note that Assumption \ref{ass:mixture} implies that 
\begin{align*}
    \ex[X_i^2] & = \sum_{z \in E} \ex[X_i^2 \mid Z = z] \P(Z = z) = 
    \sum_{z \in E} (\sigma_i^2(z) + \mu_i(z)^2)\P(Z = z) < \infty \\
    \ex\left[\Var \left[X_i \mid Z \right] \right] & = 
    \sum_{z \in E} \sigma_i^2(z) \P(Z = z)  \geq \sigma_{\min}^2 \min_{z \in E} \P(Z = z) > 0.
\end{align*}
Hence Assumption \ref{ass:mixture}, combined with $\ex[Y^2] < \infty$, ensure that 
the moment conditions in Assumption \ref{ass:moment} hold.

The following proposition states that the mixture model given by Assumption \ref{ass:mixture}
is a special case of the general latent variable model. 

\begin{proposition} \label{prop:recover}
Assumption \ref{ass:mixture} on the mixture model 
implies Assumption \ref{ass:variable}. Specifically, that $\sigma(Z) \subseteq \sigma(\bX_{-i})$ for all $i \in \nat$.
\end{proposition} 

\begin{rem} \label{rem:null-set} \rm
The proof of Proposition \ref{prop:recover} is in Appendix \ref{app:error_rate_bound}. 
Technically, the proof only gives \emph{almost sure} recovery of $Z$ from $\bX_{-i}$, and we can thus only conclude that $\sigma(Z)$ is contained in $\sigma(\bX_{-i})$ up to negligible sets. We can, however, replace $Z$ by a variable, $Z'$, such that $\sigma(Z') \subseteq \sigma(\bX_{-i})$
and $Z' = Z$ almost surely. We can thus simply swap $Z$ with $Z'$ in Assumption 
\ref{ass:mixture}.
\end{rem}

\begin{rem} \rm 
The arguments leading to Proposition \ref{prop:recover}
rely on Assumptions \ref{ass:mixture}(2) and \ref{ass:mixture}(3)---specifically 
the separation condition \eqref{eq:sep} and the upper bound in \eqref{eq:sigmabound}. However, these conditions are 
not necessary to be able to recover $Z$ from $\bX_{-i}$. Using Kakutani's 
theorem on equivalence of product measures
it is possible to characterize precisely when $Z$ can be recovered, but 
the abstract characterization is not particularly operational. In Appendix \ref{app:kakutani}
we analyze the characterization for the Gaussian mixture model, where $X_i$ given $Z = z$
has a $\mathcal{N}(\mu_i(z), \sigma^2_i(z))$-distribution. This leads to 
Proposition \ref{prop:kakutani} and Corollary \ref{cor:kakutani}
in Appendix \ref{app:kakutani}, which gives 
necessary and sufficient conditions for recovery in the Gaussian mixture model. 
\end{rem}

\subsection{Bounding estimation error due to using substitutes} \label{sec:measurement_error}
In this section we derive an upper bound on the estimation error, which is due to using substitutes, cf. the decomposition \eqref{eq:beta_hat_tilde_decomposition}. To this end, we consider the 
(partly hypothetical) observations 
$(x_{i,1}, \hat{z}_1, z_1, y_1), \ldots (x_{i,n}, \hat{z}_n, z_n, y_n)$, which include the 
otherwise unobserved $z_k$-s as well as their observed substitutes, the $\hat{z}_k$-s. We let $\bx_i = (x_{i,1}, \ldots, x_{i,n})^T \in \real^n$ and $\by = (y_1, \ldots, y_n)^T \in \real^n$, and 
$\|\bx_i\|_2$ and $\|\by\|_2$ denote the $2$-norms of $\bx_i$ and $\by$, respectively. We also 
let 
\[
n(z) = \sum_{k=1}^n 1(z_k = z) \qquad \text{and} \qquad \hat{n}(z) = \sum_{k=1}^n 1(\hat{z}_k = z) 
\]
for $z \in E = \{1, \ldots, K\}$, and  
\[ 
n_{\min} = \min \{n(1), \ldots, n(K), \hat{n}(1), \ldots, \hat{n}(K)\}.
\]
Furthermore, 
\[ 
\overline{\mu}_i(z) = \frac{1}{n(z)} \sum_{k: z_k = z} x_{i,k},
\]
and we define the following three quantities 
\begin{align}
    \alpha & = \frac{n_{\min}}{n} \label{eq:alpha} \\ 
    \delta & = \frac{1}{n} \sum_{k=1}^n 1(\hat{z}_k \neq z_k) \label{eq:delta} \\
    \rho & = \frac{\min\left\{\sum_{k=1}^{n} (x_{i,k}- \overline{\mu}_i(z_k))^2, 
    \sum_{k=1}^{n} (x_{i,k}- \hat{\mu}_i(\hat{z}_k))^2 \right\}}{\| \bx_i\|_2^2}. \label{eq:rho}
\end{align}

\begin{theorem} \label{thm:error_bound}
Let $\alpha$, $\delta$ and $\rho$ be given by \eqref{eq:alpha},
\eqref{eq:delta} and \eqref{eq:rho}. If $\alpha, \rho > 0$  then 
    \begin{equation} \label{eq:bias_bound}
        |\widehat{\beta}^{\s}_i - \hat{\beta}_i| \leq \frac{2\sqrt{2}}{\rho^2} 
        \sqrt{\frac{\delta}{\alpha}} \frac{\|\by\|_2}{\|\bx_i\|_2}.
    \end{equation}
\end{theorem}

The proof of Theorem \ref{thm:error_bound} is given in 
Appendix \ref{app:error_proof}. Appealing to the Law of Large Numbers, 
the quantities in the upper bound \eqref{eq:bias_bound} 
can be interpreted as follows:
\begin{itemize}
\item The ratio $\|\by\|_2 / \|\bx_i\|_2$ is 
approximately a fixed and finite constant (unless $X_i$ is constantly zero)
depending on the marginal distributions of $X_i$ and $Y$ only.
\item The fraction $\alpha$ is approximately 
\begin{equation} \label{eq:alph_low}
\min_{z \in E} \left\{ \min\{ \P(Z = z), \P(\hat{Z} = z)\} \right\},
\end{equation}
which is strictly positive by Assumption \ref{ass:mixture}(4) (unless 
recovery is working poorly).
    \item The quantity $\rho$ is a standardized 
measure of the residual variation of the $x_{i,k}$-s within the groups defined by 
the $z_k$-s or the $\hat{z}_k$-s. It is approximately equal to the constant  
\[
\frac{\min \left\{\ex\left[\Var \left[X_i \mid Z \right] \right], 
\ex\left[\Var \left[X_i \mid \hat{Z} \right] \right] \right\}}{E(X_i^2)},
\]
which is strictly positive if the probabilities in \eqref{eq:alph_low} are strictly positive
and not all of the conditional variances are $0$.
\item The fraction $\delta$ is the relative mislabeling frequency of the substitutes. It is approximately equal to the mislabeling rate $\P(\hat{Z} \neq Z)$.
\end{itemize}
The bound \eqref{eq:bias_bound} tells us that if the 
mislabeling rate of the substitutes tends to $0$, that is, if $\P(\hat{Z} \neq Z) \to 0$, 
the estimation error tends to 
$0$ roughly like $\sqrt{\P(\hat{Z} \neq Z)}$. This could potentially be achieved 
by letting $p \to \infty$ and $m \to \infty$. We formalize this statement in 
Section \ref{sec:asymp}.

\subsection{Bounding the mislabeling rate of the substitutes} \label{sec:error_rate_bound}
In this section we give bounds on the mislabeling rate, $\P(\hat{Z} \neq Z)$, 
with the ultimate purpose 
of controlling the magnitude of $\delta$ in the bound \eqref{eq:bias_bound}. Two
different approximations are the culprits of mislabeling. 
First, the computation of $\hat{Z}$ is based on the $p$ variables in $\bX_{1:p}$ only, and 
it is thus an approximation of the full recovery map based on all variables in $\bX$. Second, 
the recovery map is an estimate and thus itself an approximation. 
The severity of the second approximation is quantified by the following relative 
errors of the conditional means used for recovery. 

\begin{dfn}[Relative errors, $p$-separation] 
For the mixture model given by Assumption~\ref{ass:mixture} let 
$\boldsymbol{\mu}_{1:p}(z) = (\mu_i(z))_{i=1, \ldots, p} \in \real^p$ for $z \in E$. With
$\check{\boldsymbol{\mu}}_{1:p}(z) \in \real^p$ for $z \in E$ any collection of $p$-vectors, 
define the relative errors 
\begin{equation} \label{eq:relerr}
R_{z,v}^{(p)} = \frac{\|\boldsymbol{\mu}_{1:p}(z) - \check{\boldsymbol{\mu}}_{1:p}(z)\|_2}
{\|\boldsymbol{\mu}_{1:p}(z) -  \boldsymbol{\mu}_{1:p}(v)\|_2}
\end{equation}
for $z, v \in E$,  $v \neq z$. Define, moreover, the minimal $p$-separation as 
\begin{equation} \label{eq:minsep}
\sep(p) = \min_{z \neq v} \norm{\boldsymbol{\mu}_{1:p}(z) - \boldsymbol{\mu}_{1:p}(v)}_2^2.
\end{equation}
\end{dfn}

Note that Assumption \ref{ass:mixture}(2) implies that $\sep(p) \to \infty$ for 
$p \to \infty$. This convergence could be arbitrarily slow. The following 
definition captures the important case where the separation grows at least 
linearly in $p$.

\begin{dfn}[Strong separation]
We say that the mixture model satisfies \emph{strong separation} if there exists an 
$\epsilon > 0$ such that $\sep(p) \geq \epsilon p$ eventually.
\end{dfn}

Strong separation is equivalent to 
\[
\liminf_{p \to \infty} \frac{\sep(p)}{p} > 0.
\]
A sufficient condition for strong separation is that 
$|\mu_i(z) - \mu_i(v)| \geq \epsilon$ eventually for all $z,v \in E$, $v \neq z$ 
and some $\epsilon > 0$.
That is, $\liminf_{i \to \infty} |\mu_i(z) - \mu_i(v)| > 0$ for $v \neq z$.
When we have strong separation, then for $p$ large enough
\[
\left(R_{z,v}^{(p)}\right)^2 \leq \frac{1}{\epsilon p} \|\boldsymbol{\mu}_{1:p}(z) - \check{\boldsymbol{\mu}}_{1:p}(z)\|_2^2
\leq \frac{1}{\epsilon} \max_{i=1, \ldots, p} \left(\mu_{i}(z) - 
\check{\mu}_i(z) \right)^2, 
\] 
and we note that it is conceivable\footnote{Parametric assumptions, say, and 
marginal estimators of each $\mu_i(z)$  
that, under Assumption \ref{ass:mixture}, are uniformly consistent over $i \in \nat$ can be 
combined with a simple union bound to show the claim, possibly in a suboptimal way, cf. Section \ref{sec:tensor_decomp}.} that we can estimate $\boldsymbol{\mu}_{1:p}(z)$ 
by an estimator, $\check{\boldsymbol{\mu}}_{1:p}(z)$, such that for $m, p \to \infty$ 
appropriately, $R_{z,v}^{(p)} \overset{P}{\to} 0$.

The following proposition shows that a bound on $R_{z,v}^{(p)}$ is sufficient to 
ensure that the growth of $\sep(p)$ controls how fast the mislabeling rate diminishes with $p$. The proposition is stated for a fixed $\check{\boldsymbol{\mu}}$, which means that when $\check{\boldsymbol{\mu}}$ is an estimate, 
we are effectively assuming it is independent of the template observation $(\bX_{1:p}, Z)$ 
used to compute $\hat{Z}$. 

\begin{proposition} \label{prop:error_rate_bound} 
Suppose that Assumption \ref{ass:mixture} holds. Let $\check{\boldsymbol{\mu}}_{1:p}(z) \in \real^p$
for $z \in E$ and let 
\[
\hat{Z} = \argmin_{z} \| \bX_{1:p} - \check{\boldsymbol{\mu}}_{1:p}(z)\|_2.
\]
Suppose also that $R_{z,v}^{(p)} \leq \frac{1}{10}$ for all $z, v \in E$ with $v \neq z$. Then 
\begin{equation} \label{eq:gen_bound}
    \P\left(\hat{Z} \neq Z\right) \leq \frac{25 K \sigma^2_{\max}}{\sep(p)}.
\end{equation}
If, in addition, the conditional distribution of $X_i$ given $Z = z$ is sub-Gaussian 
with variance factor $v_{\max}$, independent of $i$ and $z$, then
\begin{equation}  \label{eq:sub_gauss_bound}
    \P\left(\hat{Z} \neq Z\right) \leq K \exp\left( - \frac{\sep(p)}{50 v_{\max}} \right)
\end{equation}
\end{proposition}

\begin{rem} \label{rem:constant} \rm
The proof of Proposition \ref{prop:error_rate_bound} is in Appendix \ref{app:error_rate_bound}. It shows that the specific constants, $25$ and $50$, 
appearing in the bounds above 
hinge on the 
specific bound, $R_{z,v}^{(p)} \leq \frac{1}{10}$, on the relative error. The proof works for 
any bound strictly smaller than $\frac{1}{4}$. Replacing $\frac{1}{10}$
by a smaller bound on the relative errors decreases the constant, but it will 
always be larger than $4$.    
\end{rem}

The upshot of Proposition \ref{prop:error_rate_bound} is that if the relative errors, 
$R_{z,v}^{(p)}$, are sufficiently small then Assumption \ref{ass:mixture} is sufficient 
to ensure that $\P\left(\hat{Z} \neq Z\right) \to 0$ for $p \to \infty$. Without 
additional distributional assumptions the general bound \eqref{eq:gen_bound} decays slowly 
with $p$, and even with strong separation, 
the bound only gives a rate of $\tfrac{1}{p}$. With the additional sub-Gaussian 
assumption, the rate is improved dramatically, and with strong separation it improves to 
$e^{-cp}$ for some constant $c > 0$. If the $X_i$-s are bounded, their (conditional) 
distributions are sub-Gaussian, thus the rate is fast in this special but important 
case. 

\subsection{Asymptotics of the substitute adjustment estimator} \label{sec:asymp}
Suppose $Z$ takes values in $E = \{1, \ldots, K\}$ and that 
$(x_{i, 1}, z_{1}, y_1), \ldots, (x_{i, n}, z_{n}, y_n)$ are
observations of $(X_i, Z, Y)$. Then 
Assumption \ref{ass:moment} ensures that the oracle OLS estimator
$\widehat{\beta}_i$ is $\sqrt{n}$-consistent and that
\[
\widehat{\beta}_i \overset{\text{as}}{\sim} \mathcal{N}(\beta_i, w_i^2/n).
\]
There are standard sandwich formulas for
the asymptotic variance parameter $w_i^2$. In this section we 
combine the bounds from Sections \ref{sec:measurement_error} and 
\ref{sec:error_rate_bound} to show our main theoretical result; that $\widehat{\beta}^{\s}_i$
is a consistent and asymptotically normal estimator of $\beta_i$ for $n,m \to \infty$
if also $p \to \infty$ appropriately. 

\begin{assump} \label{ass:independence} \rm
The dataset $\S_0$ in Algorithm \ref{alg:leanalgmix} consists 
of i.i.d. observations of $\bX_{1:p}$, the dataset $\S$ in Algorithm \ref{alg:leanalgmix} consists of i.i.d. observations of $(\bX_{1:p}, Y)$, and $\S$ is independent of $\S_0$.
\end{assump}

\begin{theorem} \label{thm:main}
Suppose Assumption \ref{ass:regpos} holds and $E(Y^2) < \infty$, 
and consider the mixture model fulfilling Assumption \ref{ass:mixture}. Consider
data satisfying Assumption \ref{ass:independence} and the estimator 
$\widehat{\beta}^{\s}_i$ given by Algorithm \ref{alg:leanalgmix}. Suppose that  
$n, m, p \to \infty$ such that $\P(R_{z,v}^{(p)} > \tfrac{1}{10}) \to 0$. Then 
the following hold:
\begin{enumerate}
    \item The estimation error due to using substitutes tends to $0$ in 
    probability, that is, 
\[ 
    |\widehat{\beta}^{\s}_i - \hat{\beta}_i| \overset{P}{\to} 0, 
\]
and $\widehat{\beta}^{\s}_i$ is a consistent estimator of $\beta_i$. 
\item If $\frac{\sep(p)}{n} \to \infty$ and 
$n \P(R_{z,v}^{(p)} > \tfrac{1}{10}) \to 0$, then $\sqrt{n} |\widehat{\beta}^{\s}_i - \hat{\beta}_i| \overset{P}{\to} 0$.
\item If $X_i$ conditionally on $Z = z$ is sub-Gaussian, 
with variance factor independent of $i$ and $z$, and if $\frac{\sep(p)}{\log(n)} \to \infty$
and $n \P(R_{z,v}^{(p)} > \tfrac{1}{10}) \to 0$, 
then $\sqrt{n} |\widehat{\beta}^{\s}_i - \hat{\beta}_i| \overset{P}{\to} 0$.
\end{enumerate}
In addition, in case (2) as well as case (3), $\widehat{\beta}^{\s}_i \overset{\text{as}}{\sim} \mathcal{N}(\beta_i, w_i^2/n)$, where the 
asymptotic variance parameter $w_i^2$ is the same as for the oracle estimator $\widehat{\beta}_i$.
\end{theorem}

\begin{rem} \rm 
The proof of Theorem \ref{thm:main} is in Appendix \ref{app:main_proof}. 
As mentioned in Remark \ref{rem:constant}, the precise value of the constant $\tfrac{1}{10}$ is not important. It could be replaced by any other constant \emph{strictly smaller} than 
$\tfrac{1}{4}$, and the conclusion would be the same.
\end{rem}

\begin{rem} \rm 
The general growth condition on 
$p$ in terms of $n$ in case (2) is bad; even with strong separation we 
would need $\tfrac{p}{n} \to \infty$, that is, $p$ should 
grow faster than $n$. In the sub-Gaussian case 
this improves substantially so that $p$ only needs to grow faster than $\log(n)$.
\end{rem}

\subsection{Tensor decompositions} \label{sec:tensor_decomp}
One open question from both a theoretical 
and practical perspective is how we construct the estimators $\check{\boldsymbol{\mu}}_{1:p}(z)$.
We want to ensure consistency for $m, p \to \infty$, which is expressed as 
$\P\left(R_{z,v}^{(p)} > \tfrac{1}{10}\right) \to 0$ in our theoretical results, 
and that the estimator can be computed efficiently for large $m$ and $p$. We indicated 
in Section \ref{sec:error_rate_bound} that simple marginal estimators of 
$\mu_{i}(z)$ can achieve this, but such estimators may be highly inefficient. 
In this section we briefly describe two methods based on tensor decompositions
\citep{anandkumar2014tensor} related to the third order moments of $\bX_{1:p}$. 
Thus to apply such methods we need to 
additionally assume that the $X_i$-s have finite third moments. 

Introduce first the third order $p \times p \times p$ tensor $G^{(p)}$ as
\[
G^{(p)} = \sum_{i=1}^p \mathbf{a}_i \otimes \mathbf{e}_i \otimes \mathbf{e}_i + 
\mathbf{e}_i \otimes \mathbf{a}_i \otimes \mathbf{e}_i + \mathbf{e}_i \otimes \mathbf{e}_i \otimes \mathbf{a}_i, 
\]
where $\mathbf{e}_i \in \real^p$ is the standard basis vector with a $1$ in 
the $i$-th coordinate and $0$ elsewhere, and where 
\[
\mathbf{a}_i = \sum_{z \in E} \P(Z = z) \sigma_i^2(z)\boldsymbol{\mu}_{1:p}(z). 
\]
In terms of the third order raw moment tensor and $G^{(p)}$ we define the 
tensor 
\begin{equation}
    M_3^{(p)} = \ex[\bX_{1:p} \otimes \bX_{1:p} \otimes \bX_{1:p}] - G^{(p)}.
\end{equation}
Letting $\I = \{(i_1, i_2, i_3) \in \{1, \ldots, p\} \mid i_1, i_2, i_3 \text{ all distinct}\}$
denote the set of indices of the tensors with all entries distinct, we see from 
the definition of $G^{(p)}$ that $G^{(p)}_{i_1, i_2, i_3} = 0$ for $(i_1, i_2, i_3) \in \mathcal{I}$. Thus 
\[
(M_3^{(p)})_{i_1, i_2, i_3} =  \ex\left[ X_{i_1}X_{i_2}X_{i_3} \right]
\]
for $(i_1, i_2, i_3) \in \mathcal{I}$. In the following, $(M^{(p)}_3)_{\I}$ denotes 
the incomplete tensor obtained by restricting the indices of $M^{(p)}_3$ to $\I$.

The key to using the $M_3^{(p)}$-tensor 
for estimation of the $\mu_i(z)$-s is the following rank-$K$ tensor decomposition,
\begin{equation}
    M_3^{(p)} = 
    \sum_{z = 1}^K \P(Z = z) \boldsymbol{\mu}_{1:p} (z)  \otimes 
    \boldsymbol{\mu}_{1:p}(z) \otimes \boldsymbol{\mu}_{1:p}(z);
\end{equation}
see Theorem 3.3 in \citep{anandkumar2014tensor} or the derivations on page 2 in 
\citep{guo2022incomplete_tensor}. 

\citet{guo2022incomplete_tensor} propose an algorithm based on incomplete tensor 
decomposition as follows: Let $(\widehat{M}_3^{(p)})_{\I}$ denote an estimate of the 
incomplete tensor $(M_3^{(p)})_{\I}$; obtain an 
approximate rank-$K$ tensor decomposition of the incomplete tensor $(\widehat{M}_3^{(p)})_{\I}$; extract estimates $\check{\boldsymbol{\mu}}_{1:p}(1), \ldots, \check{\boldsymbol{\mu}}_{1:p}(K)$ 
from this tensor decomposition. 
Theorem 4.2 in \citep{guo2022incomplete_tensor} shows that if the vectors $\boldsymbol{\mu}_{1:p}(1), \ldots, \boldsymbol{\mu}_{1:p}(K)$ satisfy certain regularity assumptions, they 
are estimated consistently by their algorithm (up to permutation) 
if $(\widehat{M}_3^{(p)})_{\I}$ is consistent. We note that the regularity assumptions 
are fulfilled for generic vectors in $\real^p$.

A computational downside of working directly with $M_3^{(p)}$ is that it grows 
cubically with $p$. \cite{anandkumar2014tensor} propose to consider 
$\widetilde{\bX}^{(p)} = \mathbf{W}^T \bX_{1:p} \in \real^K$, where $\mathbf{W}$ is a 
$p \times K$ whitening matrix. The tensor decomposition is then computed 
for the corresponding $K \times K \times K$ tensor $\widetilde{M}_3$. When 
$K < p$ is fixed and $p$ grows, this is computationally advantageous. Theorem 
5.1 in \citet{anandkumar2014tensor} shows that, under a generically satisfied 
non-degeneracy condition, the tensor decomposition of $\widetilde{M}_3$ can 
be estimated consistently (up to permutation) if $\widetilde{M}_3$ can be estimated consistently. 

To use the methodology from \citet{anandkumar2014tensor} in Algorithm \ref{alg:leanalgmix}, 
we replace Step~4 by their Algorithm 1 applied to $\widetilde{\bx}^{(0, p)} = \mathbf{W}^T \bx_{1:p}^{(0)}$. This will estimate the transformed mean vectors 
$\widetilde{\boldsymbol{\mu}}^{(p)}(z) = \mathbf{W}^T \boldsymbol{\mu}_{1:p}(z) \in \real^K$.
Likewise, we replace Step 5 in Algorithm \ref{alg:leanalgmix} by 
\[
\hat{z}_k = \argmin_{z} \left\| \widetilde{\bx}^{(p)}  - \check{\widetilde{\boldsymbol{\mu}}}^{(p)}(z) \right\|_2
\]
where $\widetilde{\bx}^{(p)} = \mathbf{W}^T \bx_{1:p}$. The separation and 
relative errors conditions should then be expressed in terms of the 
$p$-dependent $K$-vectors $\widetilde{\boldsymbol{\mu}}^{(p)}(1), \ldots, 
\widetilde{\boldsymbol{\mu}}^{(p)}(K) \in \real^K$.

\section{Simulation Study}
\label{sec:simulations}
Our analysis in Section \ref{sec:finite_mixture_model_full} shows that Algorithm 3 is capable 
of consistently estimating the $\beta_i$-parameters via substitute adjustment for 
$n, m, p \to \infty$ appropriately. The purpose of this section is to shed light on 
the finite sample performance of substitute adjustment via a simulation study. 

The $X_i$-s are simulated according to a mixture model fulfilling Assumption 
\ref{ass:mixture}, and the outcome model is as in Example \ref{ex:lin}, which 
makes $b_x^i(z) = \ex[Y \mid X_i = x; Z = z]$ a partially linear model. Throughout, we 
take $m = n$ and $\mathcal{S}_0 = \mathcal{S}$ in Algorithm \ref{alg:leanalgmix}.
The simulations are carried out for different choices of $n$, $p$, $\boldsymbol{\beta}$
and $\mu_i(z)$-s, and we report results on both the mislabeling rate of the 
latent variables and the mean squared error (MSE) of the $\beta_i$-estimators.

\subsection{Mixture model simulations and recovery of $Z$} \label{sec:experiments_z} 
The mixture model in our simulations is given as follows. 

\begin{itemize}
    \item We set $K = 10$ and fix $p_{\max} = 1000$ and $n_{\max} = 1000$. 
    \item We draw $\mu_i(z)$-s independently and uniformly from $(-1, 1)$ for 
$z \in \{1, \ldots, K\}$ and $i \in \{1, \ldots, p_{\max} \}$.
    \item Fixing the $\mu_i(z)$-s and a choice 
of $\mu_{\mathrm{scale}} \in \{0.75, 1, 1.5\}$, we simulate $n_{\max}$ independent 
observations of $(\bX_{1:{p_{\max}}}, Z)$, each with the latent variable
$Z$ uniformly distributed on $\{1,...,K\}$, and $X_i$ given $Z=z$ being
$\mathcal{N}(\mu_{\mathrm{scale}} \cdot \mu_i(z), 1)$-distributed.
\end{itemize}

We use the algorithm from \citet{anandkumar2014tensor}, as described in 
Section \ref{sec:tensor_decomp}, for recovery. We replicate the simulation outlined above 
$10$ times, and we consider recovery of $Z$ for $p \in \{50, 100, 200, 1000\}$ and 
$n \in \{50, 100, 200, 500, 1000\}$. 
For replication $b \in \{1, \ldots, 10\}$ the actual values of the latent variables 
are denoted $z_{b,k}$. For each combination of $n$ and $p$ 
the substitutes are denoted $\hat{z}_{b,k}^{(n,p)}$. 
The mislabeling rate for fixed $p$ and $n$ is estimated as 
\[
\delta^{(n,p)} = \frac{1}{10} \sum_{b = 1}^{10} 
\frac{1}{n} \sum_{k=1}^n 1(\hat{z}_{b,k}^{(n,p)} \neq z_{b,k}).
\]

\begin{figure}
    \centering
    \begin{subfigure}{}
        \centering
        \includegraphics[scale=.6]{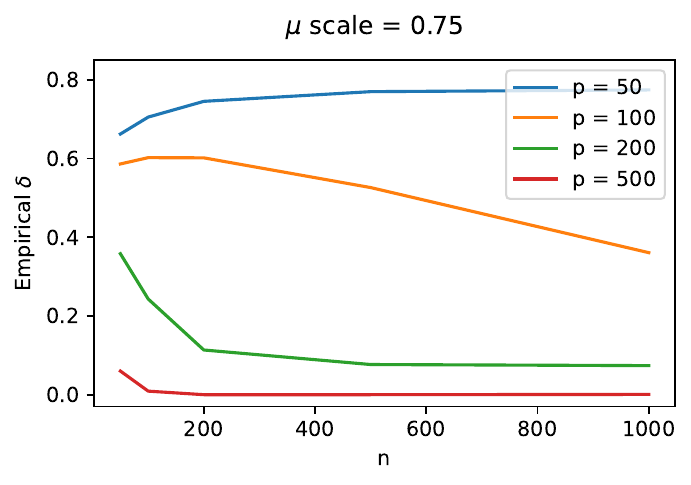}
    \end{subfigure}
    \begin{subfigure}{}
        \centering
        \includegraphics[scale=.6]{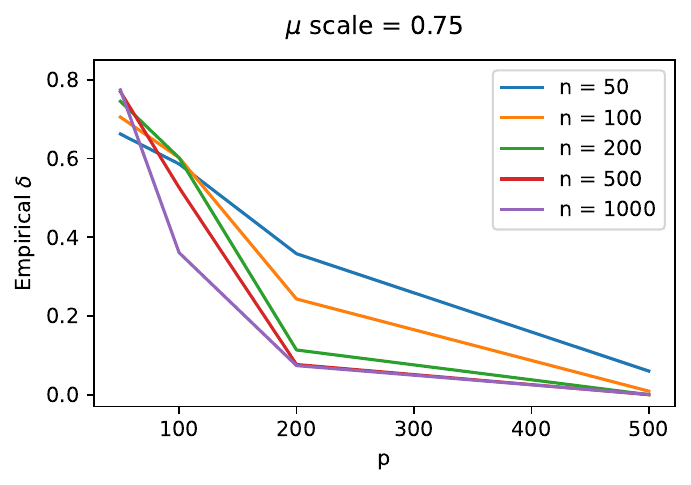}
    \end{subfigure}
    \vfill    
    \begin{subfigure}{}
        \centering
        \includegraphics[scale=.6]{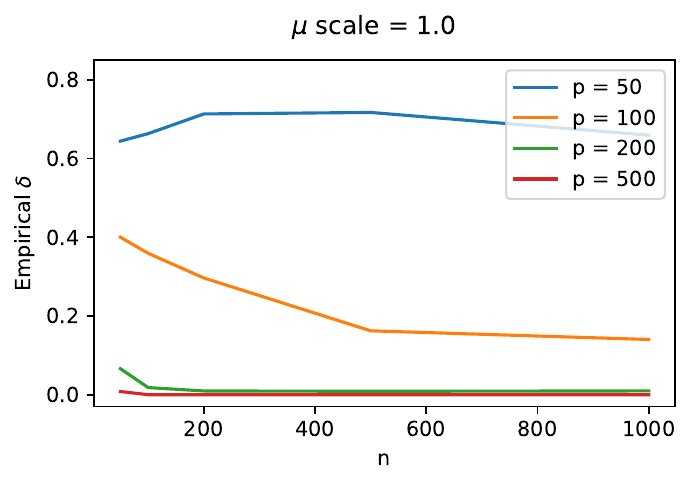}
    \end{subfigure}
    \begin{subfigure}{}
        \centering
        \includegraphics[scale=.6]{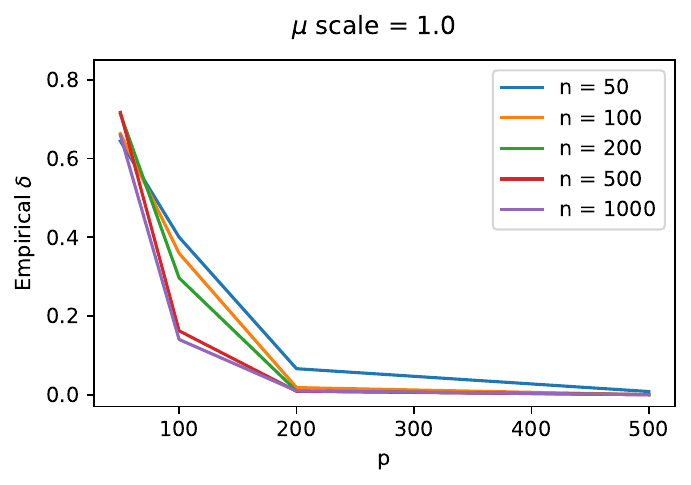}
    \end{subfigure}
    \vfill
    \begin{subfigure}{}
        \centering
        \includegraphics[scale=.6]{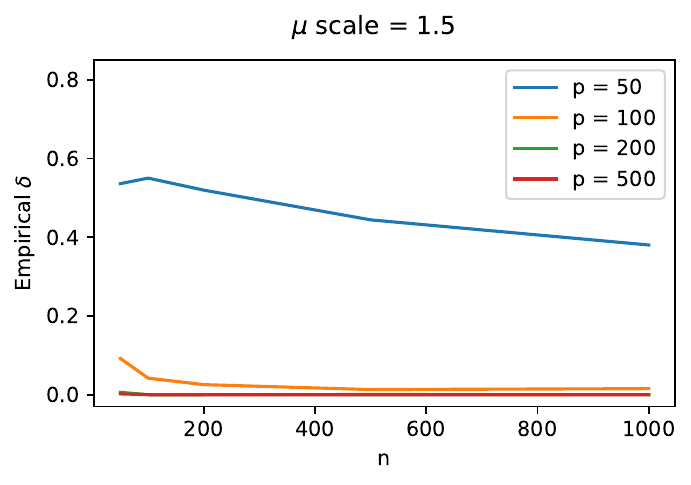}
    \end{subfigure}
    \begin{subfigure}{}
        \centering
        \includegraphics[scale=.6]{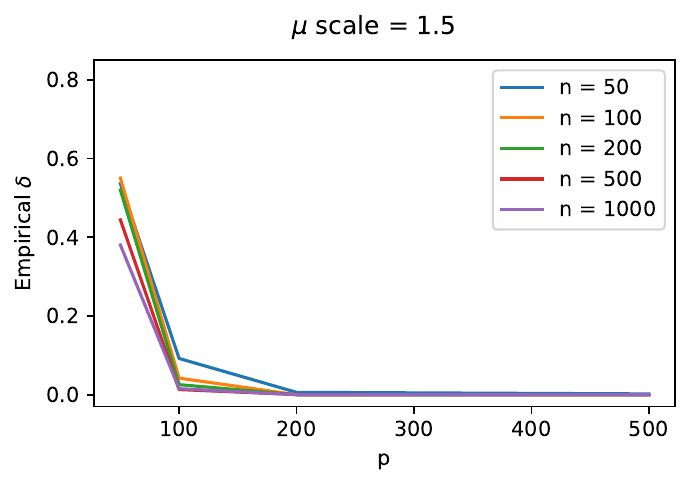}
    \end{subfigure}
    \caption{Empirical mislabeling rates as a function of $n = m$ and $p$ and for three different separation scales.}
    \label{fig:np_with_mu}
\end{figure}

Figure \ref{fig:np_with_mu} shows the estimated mislabeling rates from the simulations. 
The results demonstrate that for reasonable choices of $n$ and $p$, the 
algorithm based on \citep{anandkumar2014tensor} is capable of recovering  $Z$ 
quite well. 

The theoretical upper bounds of the mislabeling rate in Proposition 
\ref{prop:error_rate_bound} are monotonely decreasing as functions  
of $\norm{\boldsymbol{\mu}_{1:p}(z) - \boldsymbol{\mu}_{1:p}(v)}_2$. These are, 
in turn, monotonely increasing in $p$ and in $\mu_{\mathrm{scale}}$. 
The results in Figure \ref{fig:np_with_mu} support that this behavior of 
the upper bounds carry over to the actual mislabeling rate. Moreover, 
the rapid decay of the mislabeling rate with $\mu_{\mathrm{scale}}$ is in 
accordance with the exponential decay of the upper bound in 
the sub-Gaussian case.

\subsection{Outcome model simulation and estimation of $\beta_i$} \label{sec:experiments_beta}
Given simulated $Z$-s and $X_i$-s as described in Section \ref{sec:experiments_z},
we simulate the outcomes as follows.

\begin{itemize}
    \item Draw $\beta_i$ independently and uniformly from $(-1, 1)$ for $i = 1, \ldots, p_{\max}$.
    \item Fix $\gamma_{\mathrm{scale}} \in \{0, 20, 40, 100, 200\}$ and let 
    $\gamma_z = \gamma_{\mathrm{scale}} \cdot z$ for $z \in \{1, \ldots, K\}$.
    \item With $\epsilon \sim \mathcal{N}(0, 1)$ simulate $n_{\max}$ independent outcomes as
    \begin{equation*}
        Y = \sum_{i=1}^{p_{\max}} \beta_i X_i + \gamma_Z + \epsilon.
    \end{equation*}  
\end{itemize}

\begin{figure}
    \centering
    \begin{subfigure}
        \centering
        \includegraphics[scale=.6]{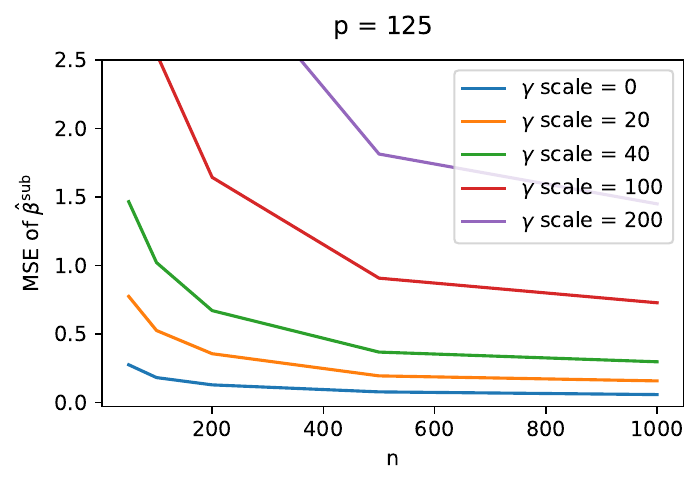}
    \end{subfigure}
    \begin{subfigure}
        \centering
        \includegraphics[scale=.6]{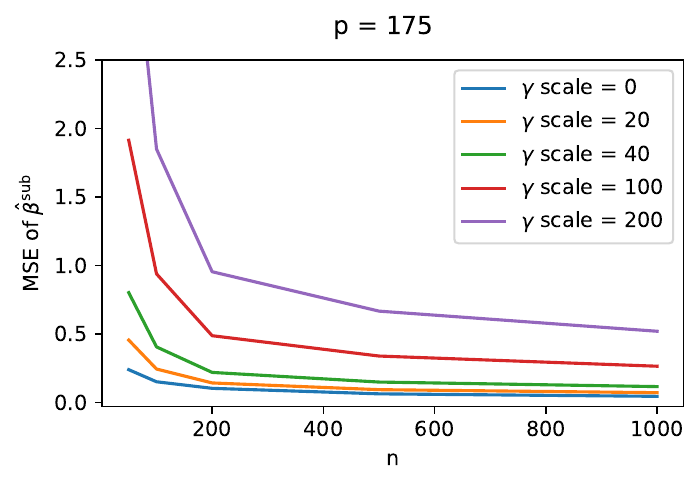}
    \end{subfigure}
    \caption{Average MSE for substitute adjustment using Algorithm \ref{alg:leanalgmix} 
    as a function of sample size $n$ and for two different dimensions, a range  of the unobserved confounding levels, and with $\mu_{\mathrm{scale}} = 1$.}
    \label{fig:n_vs_gamma}
\end{figure}

The simulation parameter $\gamma_{\mathrm{scale}}$ captures a potential effect 
of unobserved $X_i$-s for $i > p_{\max}$. We refer to this effect as \emph{unobserved 
confounding}. For $p < p_{\max}$, adjustment using the 
naive linear regression model $\sum_{i=1}^p \beta_i x_i$ would lead 
to biased estimates even if $\gamma_{\mathrm{scale}} = 0$, while the naive linear
regression model for $p = p_{\max}$ would be correct when $\gamma_{\mathrm{scale}} = 0$. 
When $\gamma_{\mathrm{scale}} > 0$, adjusting via naive linear 
regression for all observed $X_i$-s 
would still lead to biased estimates due to the unobserved confounding.

We consider the estimation error for $p \in \{125, 175\}$ and 
$n \in \{50, 100, 200, 500, 1000\}$. 
Let $\beta_{b,i}$ denote the $i$-th parameter in the $b$-th replication, 
and let $\hat{\beta}_{b,i}^{\s, n, p}$ denote the corresponding 
estimate from Algorithm \ref{alg:leanalgmix} for each combination of $n$ and $p$. The 
average MSE of $\hat{\boldsymbol{\beta}}_{b}^{\s, n, p}$  is computed as 
\[
\mathrm{MSE}^{(n,p)} = \frac{1}{10} \sum_{b=1}^{10} \frac{1}{p} \sum_{i=1}^p (\hat{\beta}_{b,i}^{\s, n, p} - 
\beta_{b,i})^2.
\]

Figure \ref{fig:n_vs_gamma} shows the MSE for the different combinations of $n$ and 
$p$ and for different choices of $\gamma_{\mathrm{scale}}$. Unsurprisingly, 
the MSE decays with sample size and increases with the magnitude of unobserved confounding. 
More interestingly, we see a clear decrease with the dimension $p$ indicating that the 
lower mislabeling rate for larger $p$ translates to a lower MSE as well.

\begin{figure}
    \centering
    \begin{subfigure}{}
        \centering
        \includegraphics[scale=.6]{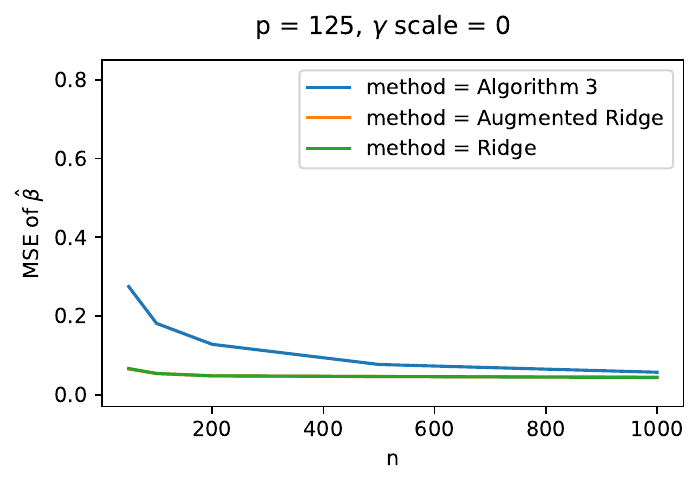}
    \end{subfigure}
    \begin{subfigure}{}
        \centering
        \includegraphics[scale=.6]{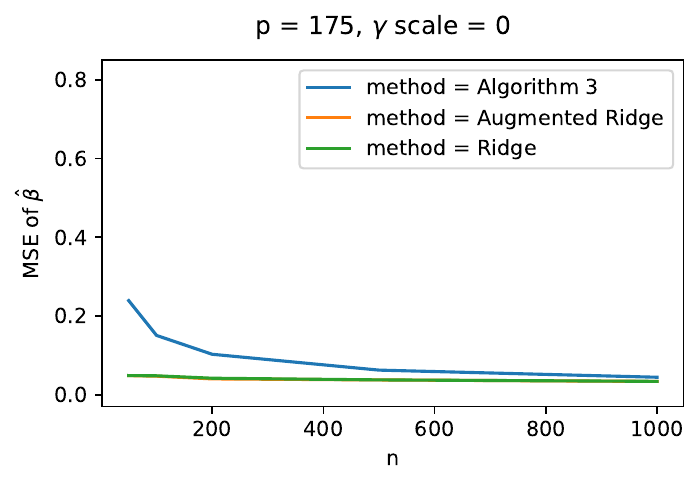}
    \end{subfigure}
    \vfill 
    \begin{subfigure}{}
        \centering
        \includegraphics[scale=.6]{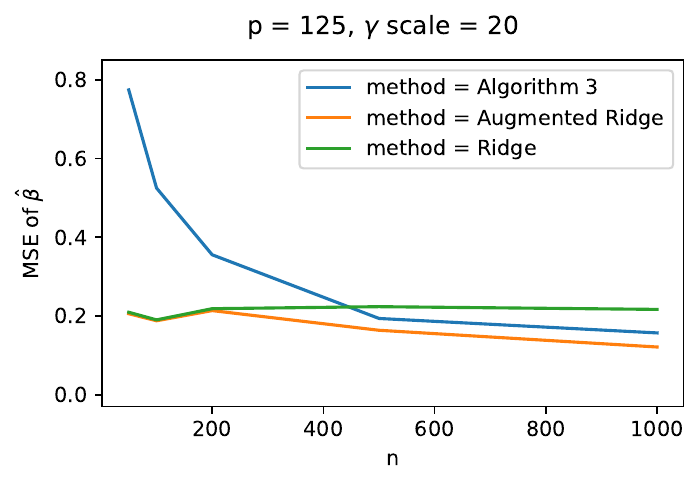}
    \end{subfigure}
    \begin{subfigure}{}
        \centering
        \includegraphics[scale=.6]{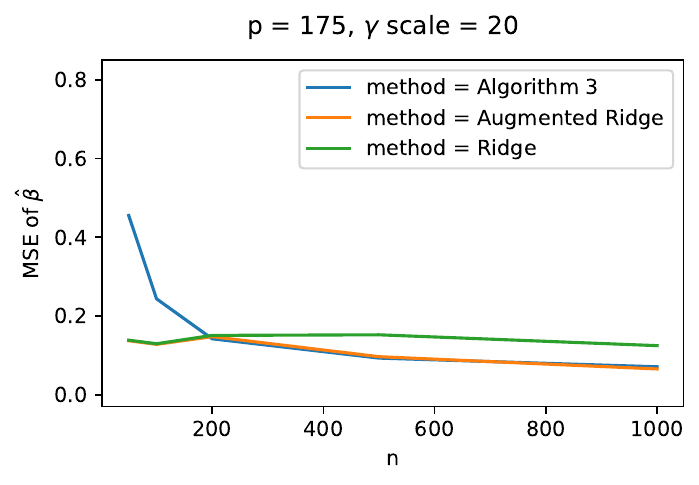}
    \end{subfigure}
    \vfill 
    \begin{subfigure}{}
        \centering
        \includegraphics[scale=.6]{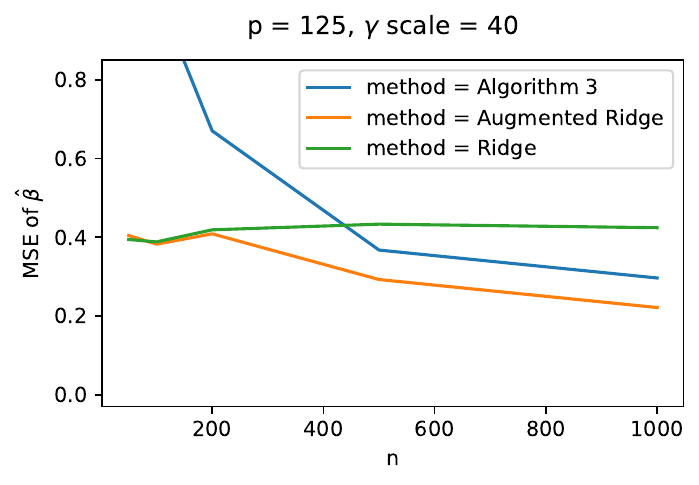}
    \end{subfigure}
    \begin{subfigure}{}
        \centering
        \includegraphics[scale=.6]{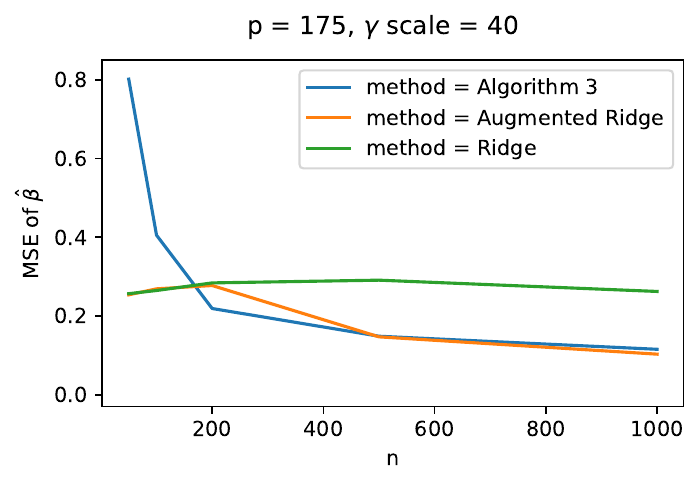}
    \end{subfigure}
    \caption{Average MSE for substitute adjustment using Algorithm \ref{alg:leanalgmix} 
    compared to average MSE for the ridge and augmented ridge estimators for 
    two different dimensions, a range of unobserved confounding levels, and with 
    $\mu_{\mathrm{scale}} = 1$.}
    \label{fig:methods_compare}
\end{figure}

Finally, we compare the results of Algorithm \ref{alg:leanalgmix} with two other approaches. 
Letting $\mathbb{X}$ denote the $n \times p$ model matrix for the $x_{i,k}$-s
and $\by$ the $n$-vector of outcomes, the ridge regression estimator is 
given as 
\begin{equation*}
    \hat{\boldsymbol{\beta}}^{(n,p)}_{\mathrm{Ridge}} = 
    \argmin_{\boldsymbol{\beta} \in \mathbb R^p} \min_{\beta_0 \in \real} \| \by - \beta_0 - \mathbb{X} \boldsymbol{\beta} \|_2^2  + \lambda \norm{\boldsymbol{\beta}}_2^2,
\end{equation*}
with $\lambda$ chosen by five-fold cross-validation. The 
augmented ridge regression estimator is given as
\begin{align*}
     \hat{\boldsymbol{\beta}}^{(n,p)}_{\text{Aug-Ridge}} = \argmin_{\boldsymbol{\beta}\in \mathbb R^p} \min_{\boldsymbol{\gamma} \in \mathbb R^K} & 
     \left\| \by - \left[ \mathbb{X}, \hat{\bZ} \right] 
    \left[ \begin{array}{c} \boldsymbol{\beta} \\ \boldsymbol{\gamma}  \end{array} \right] \right\|_2^2 
     + \lambda \norm{\boldsymbol{\beta}}_2^2, 
\end{align*}
where $\hat{\bZ}$ is the $n \times K$ model matrix of dummy variable 
encodings of the substitutes. Again, $\lambda$ is chosen by five-fold cross-validation.

The average MSE is computed for ridge regression and augmented ridge 
regression just as for substitute adjustment. Figure \ref{fig:methods_compare} 
shows results for $p=125$ and $p=175$. These two values of $p$ correspond to 
asymptotic (as $p$ stays fixed and $n\to\infty$) mislabeling rates $\delta$ 
around $7\%$ and $2\%$, respectively.

We see that both alternative estimators outperform Algorithm \ref{alg:leanalgmix} when the sample size is too small to learn $Z$ reliably. However, naive linear regression is biased, and so is ridge regression (even asymptotically), and its performance does not improve as the sample size, $n$, increases. Substitute 
adjustment as well as augmented ridge regression adjust for $\hat Z$, and their performance 
improve with $n$, despite the fact that $p$ is too small to recover $Z$ exactly. 
When $n$ and the amount of unobserved confounding is sufficiently large, both of 
these estimators outperform ridge regression. Note that it is unsurprising that the 
augmented ridge estimator performs similarly to Algorithm \ref{alg:leanalgmix} for large sample sizes, because after adjusting for the substitutes, the $x_{i,k}$-residuals are roughly 
orthogonal if the substitutes give accurate recovery, and a joint regression will give  estimates similar to those of the marginal regressions. 


We made a couple of observations (data not shown) during the simulation study. 
We experimented with changing the mixture distributions to other 
sub-Gaussian distributions as well as to the Laplace distribution and got similar results 
as shown here using the Gaussian distribution. We also implemented sample splitting, 
and though Proposition \ref{prop:error_rate_bound} assumes sample splitting, we found that the improved estimation accuracy attained by using all available data for the tensor decomposition outweighs the benefit of sample splitting in the recovery stage. 

In conclusion, our simulations show that for reasonable finite $n$ and $p$, it is 
possible to recover the latent variables sufficiently well for substitute adjustment 
to be a better alternative than naive linear or ridge regression in settings where the 
unobserved confounding is sufficiently large.

\section{Discussion} \label{sec:discussion}
We break the discussion into three parts. In the first part we revisit the discussion 
about the causal interpretation of the target parameters $\chi_x^i$ 
treated in this paper. In the second part we discuss 
substitute adjustment as a method for 
estimation of these parameters as well as the assumption-lean parameters $\beta_i$.
In the third part we 
discuss possible extensions of our results

\subsection{Causal interpretations}

The main causal question is whether a contrast of the form $\chi_x^i - \chi_{x_0}^i$ 
has a causal interpretation as an average treatment effect. The framework in
\citep{wang2019blessings} and the subsequent criticisms by \citet{damour2019multi-cause}
and \citet{ogburn2020counterexamples} are based on the $X_i$-s all being causes 
of $Y$, and on the possibility of unobserved confounding. Notably, the latent 
variable $Z$ to be recovered is not equal to an unobserved confounder, 
but \citet{wang2019blessings} argue that using the deconfounder allows us to weaken 
the assumption of ``no unmeasured confounding'' to 
``no unmeasured single-cause confounding''. The assumptions made in 
\citep{wang2019blessings} did not fully justify this claim, and we found it 
difficult to understand precisely what the causal assumptions related to $Z$ 
were. 

Mathematically precise assumptions 
that allow for identification of causal parameters from a  
finite number of causes, $X_1, \ldots, X_p$, via deconfounding are 
stated as Assumptions 1 and 2 in 
\citep{wang2020towards}.
We find these assumptions regarding 
recovery of $Z$ (also termed ``pinpointing'' in the context of the deconfounder) 
for finite $p$ implausible. Moreover, the entire framework 
of the deconfounder rests on the 
causal assumption of ``weak unconfoundedness'' in Assumption 1 and 
Theorem 1 of \citep{wang2020towards}, which might be needed for a causal 
interpretation but is unnecessary for the deconfounder algorithm to 
estimate a meaningful target parameter. 

We find it beneficial to disentangle the causal interpretation from the 
definition of the target parameter. By defining the target parameter entirely 
in terms of the observational distribution of observed (or, at least, observable)
variables, we can discuss the properties of the statistical method of 
substitute adjustment without making causal claims. We have shown
that substitute adjustment under our Assumption \ref{ass:variable} on 
the latent variable model targets 
the adjusted mean irrespectively of any unobserved confounding. 
\cite{Grimmer:2023} present a similar view. The contrast $\chi_x^i - \chi_{x_0}^i$
might have a causal interpretation in specific applications, but 
substitute adjustment as a statistical method does not rely on such an 
interpretation or assumptions needed to justify such an interpretation. 
In any specific application with multiple causes and potential 
unobserved confounding, substitute adjustment might be a useful 
method for deconfounding, but depending on the context and the causal 
assumptions we are willing to make, other methods could be 
preferable \citep{Wang:2023b}.

\subsection{Substitute adjustment: interpretation, merits and deficits}
We define the target parameter as an adjusted mean when adjusting for 
an \emph{infinite} number of variables. Clearly, this is a mathematical 
idealization of adjusting for a large number of variables, 
but it also has some important technical consequences. 
First, the recovery Assumption \ref{ass:variable}(2) is a more 
plausible modelling assumption than recovery from a finite number of 
variables. Second, it gives a clear qualitative difference 
between the adjusted mean of one (or any finite number of) variables
and regression on all variables. Third, the natural requirement in 
Assumption \ref{ass:variable}(2) that $Z$ can be recovered from $\bX_{-i}$
for any $i$ replaces the minimality of a ``multi-cause separator''
from \citep{wang2020towards}. Our assumption is that $\sigma(Z)$ is 
sufficiently minimal in a very explicit way, which ensures that $Z$ does 
not contain information unique to any single $X_i$. 

\cite{Grimmer:2023} come to a similar conclusion as we do: that the target parameter 
of substitute adjustment (and the deconfounder) is the adjusted mean $\chi^i_x$,
where you adjust for an infinite number of variables. They argue forcefully that 
substitute adjustment, using a finite number $p$ of variables, does not  
have an advantage over naive regression, that is, over estimating 
the regression function $\ex\left[Y \mid X_1 = x_1, \ldots, X_p = x_p \right]$
directly. With $i = 1$, say, they argue that substitute adjustment 
is effectively assuming a partially linear, 
semiparametric regression model 
\[
\ex\left[Y \mid X_1 = x_1, \ldots, X_p = x_p \right] 
= \beta_0 + \beta_1 x_1 + h(x_2, \ldots, x_p), 
\]
with the specific constraint that $h(x_2, \ldots, x_p) = 
g(\hat{z}) = g(f^{(p)}(x_2, \ldots, x_p))$. We 
agree with their analysis and conclusion; substitute adjustment 
is implicitly a way of making assumptions about $h$. It is also a way to leverage those 
assumptions, either by shrinking the bias compared to directly 
estimating a misspecified (linear, say) $h$, or by 
improving efficiency over methods that use a too flexible model of $h$.
We believe there is room for further studies of such bias and 
efficiency tradeoffs. 

We also believe that there are two potential benefits of substitute 
adjustment, which are not brought forward by \cite{Grimmer:2023}. First, the latent variable model can be estimated without access to 
outcome observations. This means that the inner part of $h = g \circ f^{(p)}$
could, potentially, be estimated very accurately on the basis of a large sample 
$\mathcal{S}_0$ in cases where it would be difficult to estimate the 
composed map $h$ accurately from $\mathcal{S}$ alone. Second, when $p$ is 
very large, e.g., in the millions, but $Z$ is low-dimensional, there can 
be huge computational advantages to running $p$ small parallel regressions 
compared to just one naive linear regression of $Y$ on all of $\bX_{1:p}$, 
let alone $p$ naive partially linear regressions.

\subsection{Possible extensions}
We believe that our error bound in Theorem \ref{thm:error_bound} is 
an interesting result, which in a precise way bounds the error of an OLS estimator 
in terms of errors in the regressors. This result is closely related to the 
classical literature on errors-in-variables models (or measurement error models) \citep{Durbin:1954, Cochran:1968, Schennach:2016}, though this literature focuses on 
methods for bias correction when the errors are non-vanishing. We see two 
possible extensions of our result. For one, Theorem \ref{thm:error_bound} could 
easily be generalized to $E = \real^d$. In addition, it might 
be possible to apply the bias correction techniques 
developed for errors-in-variables to improve the finite sample 
properties of the substitute adjustment estimator. 

Our analysis of the recovery error could also be extended. The concentration 
inequalities in Section \ref{sec:error_rate_bound} are 
unsurprising, but developed to match our specific needs for a 
high-dimensional analysis with as few assumptions as possible. 
For more refined results on finite mixture estimation see, e.g., \citep{Heinrich:2018},
and see \citep{Ndaoud:2022} for optimal recovery when $K = 2$ and 
the mixture distributions are Gaussian. In cases where the mixture 
distributions are Gaussian, it is also plausible that specialized algorithms such as \citep{kalai_moikra2012disentangling,gandhi_REU2016moment} are more efficient
than the methods we consider based on conditional means only. 

One general concern with substitute adjustment is model misspecification. 
We have done our analysis with minimal distributional assumptions, 
but there are, of course, two fundamental assumptions: the assumption of 
conditional independence of the $X_i$-s given the latent variable $Z$, 
and the assumption that $Z$ takes values in a finite set of size $K$. An important 
extension of our results is to study robustness to violations of these two 
fundamental assumptions. We have also not considered estimation of $K$, and it would 
likewise be relevant to understand how that affects the substitute 
adjustment estimator. 

\section*{Acknowledgments}
We thank Alexander Mangulad Christgau for helpful input. 
JA and NRH were supported by a research grant (NNF20OC0062897) from Novo Nordisk Fonden. JA also received funding from the European Union's Horizon 2020 research and innovation programme under Marie Skłodowska-Curie grant agreement No 801199.

\appendix

\section{Proofs and auxiliary results} \label{app:proofs}

\subsection{Proofs of results in Section \ref{sec:model_general_form}} \label{app:model_proofs}

\begin{proof}[Proof of Proposition \ref{prop:variablerep}]
Since $X_i$ as well as $\bX_{-i}$ take values in Borel spaces, 
there exists a regular conditional distribution given $Z = z$ of each 
\citep[Theorem 8.5]{Kallenberg:2021}. These 
are denoted $P^i_z$ and $P^{-i}_z$, respectively. Moreover, 
Assumption \ref{ass:variable}(2) and
the Doob-Dynkin lemma \citep[Lemma 1.14]{Kallenberg:2021}
imply that for each $i \in \nat$ there is a measurable map 
$f_i : \real^\nat \to E$ such that $Z = f_i(\bX_{-i})$. This implies 
that $P^{-i}(B) = \int P^{-i}_z(B) P^Z(\mathrm{d}z)$ for 
$B \subseteq \real^\nat$ measurable.

Since $Z = f_i(\bX_{-i})$ it holds that $f_i(P^{-i}) = P^Z$, and furthermore 
that $P^{-i}_z(f^{-1}_i(\{z\})) = 1$. 
Assumption \ref{ass:variable}(1) implies that $X_i$ and $\bX_{-i}$ are conditionally 
independent given $Z$, thus for $A, C \subseteq \real$ and $B \subseteq E$ 
measurable sets and $\tilde{B} = f_i^{-1}(B) \subseteq \real^\nat$, 
\begin{align*}
    \P(X_i \in A, Z \in B, Y \in C) & = \P(X_i \in A, \bX_{-i} \in \tilde{B}, Y \in C) \\
    & = \int  1_A(x) 1_{\tilde{B}}(\bx) P_{x,\bx}^i(C) 
    P(\mathrm{d} x, \mathrm{d} \bx) \\
    & = \int 1_A(x) 1_{\tilde{B}}(\bx) \, P_{x,\bx}^i(C) 
    \int P_z^i \otimes P_z^{-i}(\mathrm{d} x, \mathrm{d} \bx) P^Z(\mathrm{d}z) \\
    & = \iiint  1_A(x) 1_{\tilde{B}}(\bx) \, P_{x,\bx}^i(C) 
    P_z^i(\mathrm{d} x) P_z^{-i}(\mathrm{d} \bx) P^Z(\mathrm{d}z) \\
    & = \iiint  1_A(x) 1_{B}(z) \, \int P_{x,\bx}^i(C) P_z^{-i}(\mathrm{d} \bx)  P_z^i(\mathrm{d} x)  P^Z(\mathrm{d}z) \\
    & = \iint  1_A(x) 1_{B}(z) \, Q^i_{x,z}(C)  P_z^i(\mathrm{d} x)  P^Z(\mathrm{d}z). \\
\end{align*}
Hence $Q^i_{x,z}$ is a regular conditional distribution 
of $Y$ given $(X_i, Z) = (x, z)$. 

We finally find that 
\begin{align*}
    \chi_x^i & = \iint y \, P^i_{x, \bx} (\mathrm{d} y)
    P^{-i} (\mathrm{d}\bx) \\ 
    & = \iiint y \, P^i_{x, \bx} (\mathrm{d} y)
    P^{-i}_z (\mathrm{d}\bx) P^Z(\mathrm{d}z) \\
    & = \iint  y \, \int P^i_{x, \bx} (\mathrm{d} y)
    P^{-i}_z (\mathrm{d}\bx) P^Z(\mathrm{d}z) \\
    & = \iint  y \, Q^i_{z, \bx} (\mathrm{d} y) P^Z(\mathrm{d}z). \\
\end{align*}
\end{proof}

\begin{proof}[Proof of Proposition \ref{prop:lean-rep}] 
We find that  
\begin{align*}
        \Cov \left[X_i, Y \mid Z \right] 
    & = \ex \left[(X_i - \ex[X_i \mid Z]) Y \mid Z \right] \\
    & = \ex \left[ \ex  \left[(X_i - \ex[X_i \mid Z]) Y \mid X_i, Z \right] \mid Z \right] \\
    & = \ex \left[ (X_i - \ex[X_i \mid Z]) \ex \left[Y \mid X_i, Z \right] \mid Z \right] \\
    & = \ex \left[ (X_i - \ex[X_i \mid Z]) b_{X_i}^i(Z) \mid Z \right] \\
    & = \Cov \left[X_i, b^i_{X_i}(Z) \mid Z \right],
\end{align*}
which shows \eqref{eq:beta-i-rep}. From this representation, if $b_x^i(z) = b^i(z)$ 
does not depend on $x$, $b^i(Z)$ is $\sigma(Z)$-measurable and 
$\Cov \left[X_i, b^i(Z) \mid Z \right] = 0$, whence $\beta_i = 0$. 

If $b^i_{x}(z) =  \beta_i'(z) x + \eta_{-i}(z)$, 
\[\Cov \left[X_i, b^i_{X_i}(Z) \mid Z \right] 
= \Cov \left[X_i, \beta_i'(Z) X_i + \eta_{-i}(Z) \mid Z \right] 
= \beta_i'(Z) \Var \left[X_i \mid Z \right],\]
and \eqref{eq:beta-i-rep-lin} follows. 
\end{proof}

\subsection{Auxiliary results related to Section \ref{sec:measurement_error} and proof of Theorem \ref{thm:error_bound}} \label{app:error_proof}

Let $\bZ$ denote the $n \times K$ matrix of dummy variable encodings of the 
$z_k$-s, and let $\hat{\bZ}$ denote the similar matrix for the substitutes $\hat{z}_k$-s.
With $P_{\bZ}$ and $P_{\hat \bZ}$ the orthogonal projections onto the column spaces of 
$\bZ$ and $\hat \bZ$, respectively, we can write the estimator from Algorithm \ref{alg:leanalgmix}
as 
\begin{equation}  \label{eq:betahat_subz}
    \widehat{\beta}^{\s}_i = \frac{\langle \bx_i - P_{\hat \bZ} \bx_i, \by - P_{\hat \bZ} \by \rangle}
    {\| \bx_i  - P_{\hat \bZ} \bx_i \|_2^2}.
\end{equation}
Here $\bx_i, \by \in \real^n$ denote the $n$-vectors of $x_{i,k}$-s and $y_k$-s, respectively, 
and $\langle \cdot, \cdot \rangle$ is the standard inner product on $\real^n$, so that, e.g., 
$\| \by \|_2^2 = \langle \by, \by \rangle$. 
The estimator, had we observed the latent variables, is similarly given as 
\begin{equation} \label{eq:betahat_z}
    \hat{\beta}_i = \frac{\langle \bx_i - P_{\bZ} \bx_i, \by - P_{\bZ} \by \rangle}
    {\| \bx_i - P_{\bZ} \bx_i \|_2^2}.
\end{equation}
The proof of Theorem \ref{thm:error_bound} is based on the following bound on 
the difference between the projection matrices.

\begin{lemma} \label{lem:proj_norm} Let $\alpha$ and $\delta$ be as defined by \eqref{eq:alpha} and \eqref{eq:delta}. If $\alpha > 0$ it holds that 
\begin{equation} \label{eq:proj_bound}
    \|P_{\bZ} - P_{\hat \bZ}\|_2 \leq \sqrt{\frac{2\delta}{\alpha}},
\end{equation}
where $\| \cdot \|_2$ above denotes the operator $2$-norm also known as the spectral norm.
\end{lemma}

\begin{proof} When $\alpha > 0$, the matrices $\bZ$ and $\hat{\bZ}$ have full rank $K$. 
Let $\bZ^{+} = (\bZ^T\bZ)^{-1} \bZ^T$ and 
$\hat{\bZ}^{+} = (\hat{\bZ}^T\hat{\bZ})^{-1} \hat{\bZ}^T$ denote the Moore-Penrose inverses 
of $\bZ$ and $\hat \bZ$, respectively. Then $P_{\bZ} = \bZ \bZ^{+}$ and  $P_{\hat \bZ} = 
\hat{\bZ} \hat{\bZ}^{+}$. By Theorems 2.3 and 2.4 in \citep{stewart1977},
\[
 \|P_{\bZ} - P_{\hat \bZ}\|_2 \leq \min\left\{\|\bZ^{+}\|_2, \|\hat{\bZ}^{+}\|_2\right\} \ \|\bZ - \hat{\bZ}\|_2.
\]
The operator $2$-norm $\|\bZ^{+}\|_2$ is the square root of the largest eigenvalue of 
\[
(\bZ^T\bZ)^{-1} = \left( \begin{array}{cccc}
n(1)^{-1} & 0 & \ldots & 0 \\
0 & n(2)^{-1} & \ldots & 0 \\
\vdots & \vdots & \ddots & \vdots \\
0 & 0 & \ldots & n(K)^{-1} 
\end{array} \right).
\]
Whence $\|\bZ^{+}\|_2 \leq (n_{\min})^{-1/2} = (\alpha n)^{-1/2}$. The same bound is 
obtained for $\|\hat{\bZ}^{+}\|_2$, which gives 
\[
 \|P_{\bZ} - P_{\hat \bZ}\|_2 \leq \frac{1}{\sqrt{\alpha n}} \ \|\bZ - \hat{\bZ}\|_2.
\]
We also have that 
\[
\|\bZ - \hat{\bZ}\|_2^2 \leq \|\bZ - \hat{\bZ}\|_F^2 = \sum_{k=1}^n \sum_{i=1}^p 
(\bZ_{k,i} - \hat{\bZ}_{k,i})^2 = 2\delta n,
\]
because $\sum_{i=1}^p (\bZ_{k,i} - \hat{\bZ}_{k,i})^2 = 2$ precisely for those $k$ with $\hat{z}_k \neq z_k$ and $0$ otherwise. Combining the inequalities gives \eqref{eq:proj_bound}.
\end{proof}

Before proceeding with the proof of Theorem \ref{thm:error_bound}, note that 
\[
\sum_{k=1}^n (x_{i,k} - \overline{\mu}_i(z_k))^2 = \| \bx_i - P_{\bZ} \bx_i \|_2^2
= \| (I - P_{\bZ}) \bx_i \|_2^2 \leq \| \bx_i \|_2^2
\]
since $(I - P_{\bZ})$ is a projection. Similarly, 
$\sum_{k=1}^n (x_{i,k} - \hat{\mu}_i(\hat{z}_k))^2 = 
\| \bx_i - P_{\hat \bZ} \bx_i \|_2^2 \leq \|\bx\|_2^2$, thus 
\[
\rho = \frac{\min\left\{\| \bx_i - P_{\bZ} \bx_i \|_2^2, 
    \| \bx_i - P_{\hat \bZ} \bx_i \|_2^2 \right\}}{\| \bx_i\|_2^2} \leq 1.
\]

\begin{proof}[Proof of Theorem \ref{thm:error_bound}] First note that  since $I - P_{\hat \bZ}$ is an orthogonal projection, 
\[
\langle \bx_i - P_{\hat \bZ} \bx_i, \by - P_{\hat \bZ} \by \rangle 
= \langle \bx_i, (I - P_{\hat \bZ}) \by \rangle
\] and similarly for the other inner product in \eqref{eq:betahat_z}. Moreover, 
\[
\langle \bx_i, (I - P_{\hat \bZ}) \by \rangle - \langle \bx_i , (I - P_{\bZ}) \by \rangle
= \langle \bx_i, (P_{\bZ} - P_{\hat \bZ}) \by \rangle
\]
and 
\[
\| (I - P_{\bZ}) \bx_i \|_2^2 - \| (I - P_{\hat \bZ}) \bx_i \|_2^2 
= \| (P_{\hat \bZ} - P_{\bZ}) \bx_i \|_2^2.
\]
We find that 
\begin{align*}
    \widehat{\beta}^{\s}_i - \hat{\beta}_i & =
    \frac{\langle \bx_i, (I - P_{\hat \bZ}) \by \rangle}
    {\| (I - P_{\hat \bZ}) \bx_i \|_2^2} -
    \frac{\langle \bx_i , (I - P_{\bZ}) \by \rangle}
    {\| (I - P_{\bZ}) \bx_i \|_2^2} \\
    & =  \langle \bx_i, (I - P_{\hat \bZ}) \by \rangle \left(\frac{1}
    {\| (I - P_{\hat \bZ}) \bx_i \|_2^2} - \frac{1}
    {\| (I - P_{\bZ}) \bx_i \|_2^2} \right) \\
    & \qquad \qquad 
     + \frac{\langle \bx_i, (I - P_{\hat \bZ}) \by \rangle - \langle \bx_i , (I - P_{\bZ}) \by \rangle}
    {\| (I - P_{\bZ}) \bx_i \|_2^2} \\
    & =  \langle \bx_i, (I - P_{\hat \bZ}) \by \rangle \left( 
    \frac{\| (P_{\hat \bZ} - P_{\bZ}) \bx_i \|_2^2}
    {\| (I - P_{\hat \bZ}) \bx_i \|_2^2 \| (I - P_{\bZ}) \bx_i \|_2^2} \right) \\
    & \qquad \qquad 
     + \frac{\langle \bx_i, (P_{\bZ} - P_{\hat \bZ}) \by \rangle}
    {\| (I - P_{\bZ}) \bx_i \|_2^2}.
\end{align*}
This gives the following inequality, using that $\rho \leq 1$, 
\begin{align*}
    |\widehat{\beta}^{\s}_i - \hat{\beta}_i| & \leq 
    \frac{\|P_{\bZ} - P_{\hat \bZ}\|_2 \|\bx_i\|_2^3 \|\by\|_2}{\rho^2 \|\bx_i\|^4_2}
    + \frac{\|P_{\bZ} - P_{\hat \bZ}\|_2 \|\bx_i\|_2 \|\by\|_2}{\rho \|\bx_i\|^2_2} \\
    & = \left(\frac{1}{\rho^2} + \frac{1}{\rho} \right) \|P_{\bZ} - P_{\hat \bZ}\|_2 \frac{\|\by\|_2}{\|\bx_i\|_2} \\
    & \leq \frac{2}{\rho^2} \|P_{\bZ} - P_{\hat \bZ}\|_2 \frac{\|\by\|_2}{\|\bx_i\|_2}.
\end{align*}    
Combining this inequality with \eqref{eq:proj_bound} gives \eqref{eq:bias_bound}.
\end{proof}

\subsection{Auxiliary concentration inequalities. Proofs of Propositions  \ref{prop:recover} 
and \ref{prop:error_rate_bound}} \label{app:error_rate_bound}

\begin{lemma} \label{lem:chebyshev}
Suppose that Assumption \ref{ass:mixture} holds. Let $\check{\boldsymbol{\mu}}_{1:p}(z) \in \real^p$
for $z \in E$ and let $\hat{Z} = \argmin_{z} \| \bX_{1:p} - \check{\boldsymbol{\mu}}_{1:p}(z)\|_2$. Suppose 
that $R_{z,v}^{(p)} \leq \frac{1}{10}$ for all $z, v \in E$ with $v \neq z$ then 
\begin{equation} \label{eq:chebmis}
\P(\hat Z=v \mid Z=z) \leq \frac{25 \sigma_{\max}^2}
{\| \boldsymbol{\mu}_{1:p}(z) - \boldsymbol{\mu}_{1:p}(v)\|_2^2}.
\end{equation}
\end{lemma}

\begin{proof} Since $p$ is fixed throughout the proof, we simplify the notation by 
dropping the $1\!\! : \!\! p$ subscript and use, e.g., $\bX$ and $\boldsymbol{\mu}$ to denote 
the $\real^p$-vectors $\bX_{1:p}$ and $\boldsymbol{\mu}_{1:p}$, respectively.

Fix also $z, v \in E$ with $v \neq z$ and observe first that 
\begin{align*}
    (\hat{Z} = v) 
    &\subseteq 
        \left(\norm{\bX - \check{\boldsymbol{\mu}}(v)}_2 < \norm{\bX - \check{\boldsymbol{\mu}}(z)}_2 \right)\\
    &= \left(\langle \bX - \check{\boldsymbol{\mu}}(z), \check{\boldsymbol{\mu}}(z) - \check{\boldsymbol{\mu}}(v) \rangle < -\tfrac{1}{2}\norm{\check{\boldsymbol{\mu}}(z) - \check{\boldsymbol{\mu}}(v)}^2_2 \right) \\
    &= \Big(\langle \bX - \boldsymbol{\mu}(z), \check{\boldsymbol{\mu}}(z) - \check{\boldsymbol{\mu}}(v) \rangle <
    \\ & \qquad \qquad  -\left( \tfrac{1}{2}\norm{\check{\boldsymbol{\mu}}(z) - \check{\boldsymbol{\mu}}(v)}^2_2 
    + \langle \boldsymbol{\mu}(z) - \check{\boldsymbol{\mu}}(z), \check{\boldsymbol{\mu}}(z) - \check{\boldsymbol{\mu}}(v) \rangle\right) \Big).
\end{align*}
The objective is to bound the probability of the event above using Chebyshev's inequality.
To this end, we first use the Cauchy-Schwarz inequality to get 
\begin{align*}
    \tfrac{1}{2}\norm{\check{\boldsymbol{\mu}}(z) - \check{\boldsymbol{\mu}}(v)}^2_2 
    & + \langle \boldsymbol{\mu}(z) - \check{\boldsymbol{\mu}}(z), \check{\boldsymbol{\mu}}(z) - \check{\boldsymbol{\mu}}(v) \rangle  \\
    & \geq \tfrac{1}{2}\norm{\check{\boldsymbol{\mu}}(z) - \check{\boldsymbol{\mu}}(v)}^2_2 
    - \| \boldsymbol{\mu}(z) - \check{\boldsymbol{\mu}}(z) \|_2 \|\check{\boldsymbol{\mu}}(z) - \check{\boldsymbol{\mu}}(v) \|_2 \\
    & = \norm{\boldsymbol{\mu}(z) - \boldsymbol{\mu}(v)}_2^2 \left(\tfrac{1}{2} B_{z,v}^2 - R_{z,v}^{(p)} B_{z,v} \right),
\end{align*}
where 
\[
B_{z,v} = \frac{\norm{\check{\boldsymbol{\mu}}(z) - \check{\boldsymbol{\mu}}(v)}_2}
{\norm{\boldsymbol{\mu}(z) - \boldsymbol{\mu}(v)}_2}.
\]
The triangle and reverse triangle inequality give that 
\begin{align*}
\norm{\check{\boldsymbol{\mu}}(z) - \check{\boldsymbol{\mu}}(v)}_2
& \leq \norm{\boldsymbol{\mu}(z) - \boldsymbol{\mu}(v)}_2
+ \norm{\check{\boldsymbol{\mu}}(z) - \boldsymbol{\mu}(z)}_2 
+ \norm{\boldsymbol{\mu}(v) - \check{\boldsymbol{\mu}}(v)}_2 \\
\norm{\check{\boldsymbol{\mu}}(z) - \check{\boldsymbol{\mu}}(v)}_2 
& \geq \Big| \norm{\boldsymbol{\mu}(z) - \boldsymbol{\mu}(v)}_2 - 
\norm{\boldsymbol{\mu}(z) - \check{\boldsymbol{\mu}}(z)}_2 
- \norm{\boldsymbol{\mu}(v) - \check{\boldsymbol{\mu}}(v)}_2
 \Big|,
\end{align*}
and dividing by $\norm{\boldsymbol{\mu}(z) - \boldsymbol{\mu}(v)}_2$ 
combined with the bound $\tfrac{1}{10}$ on the relative errors yield
\begin{align*}
B_{z,v} & \leq 1 + R_{z,v}^{(p)} + R_{v,z}^{(p)} \leq \frac{6}{5}, \\
B_{z,v} & \geq \Big|1 - R_{z,v}^{(p)} - R_{v,z}^{(p)} \Big| \geq \frac{4}{5}.
\end{align*}

This gives 
\[
\tfrac{1}{2} B_{z,v}^2 - R_{z,v}^{(p)} B_{z,v} 
\geq \tfrac{1}{2} B_{z,v}^2 - \tfrac{1}{10} B_{z,v}  \geq \tfrac{6}{25}
\]
since the function $b \mapsto b^2 - \tfrac{2}{10} b$ is increasing 
for $b \geq \tfrac{4}{5}$.

Introducing the variables 
$W_i = (X_i - \mu_i(z))(\check \mu_i(z) - \check \mu_i(v))$ we 
conclude that 
\begin{align} \label{eq:Zinclusion}
    (\hat{Z} = v) 
    & \subseteq \left( 
    \sum_{i=1}^p W_i < - \tfrac{6}{25} \norm{\boldsymbol{\mu}(z) - \boldsymbol{\mu}(v)}_2^2 \right).
\end{align}
Note that $ \ex[W_i \mid Z=z] = 0$ and  
$\Var [W_i \mid Z=z] = (\check \mu_i(z) - \check \mu_i(v))^2 \sigma^2_i(z)$, 
and by Assumption \ref{ass:mixture}, the $W_i$-s are conditionally independent 
given $Z = z$, so Chebyshev's inequality gives that 
\begin{align*}
    \P(\hat Z = v \mid Z = z) & \leq \P\left(
    \sum_{i=1}^p W_i < - \tfrac{6}{25} \norm{\boldsymbol{\mu}(z) - \boldsymbol{\mu}(v)}_2^2 
    \given  Z = z \right) \\
    & \leq \left(\frac{25}{6}\right)^2  
    \frac{\sum_{i=1}^p (\check \mu_i(z) - \check \mu_i(v))^2 \sigma^2_i(z)}
    {\norm{\boldsymbol{\mu}(z) - \boldsymbol{\mu}(v)}_2^4} \\
    & \leq \left(\frac{25}{6}\right)^2  
    \frac{\sigma^2_{\max} \norm{\check{\boldsymbol{\mu}}(z) - \check{\boldsymbol{\mu}}(v)}_2^2}
    {\norm{\boldsymbol{\mu}(z) - \boldsymbol{\mu}(v)}_4^2} \\
    & \leq \left(\frac{25}{6}\right)^2 B_{z,v}^2
    \frac{\sigma^2_{\max}} 
    {\norm{\boldsymbol{\mu}(z) - \boldsymbol{\mu}(v)}_2^2} \\
    & \leq 
    \frac{25 \sigma^2_{\max}}
    {\norm{\boldsymbol{\mu}(z) - \boldsymbol{\mu}(v)}_2^2},
\end{align*}
where we, for the last inequality, used that $ B_{z,v}^2 \leq \left(\frac{6}{5}\right)^2$.
\end{proof}

Before proceeding to the concentration inequality for sub-Gaussian distributions, we 
use Lemma \ref{lem:chebyshev} to prove Proposition \ref{prop:recover}.

\begin{proof}[Proof of Proposition \ref{prop:recover}] 
Suppose that $i = 1$ for convenience. 
We take $\check{\boldsymbol{\mu}}_{1:p}(z) = \boldsymbol{\mu}_{1:p}(z)$ for all $p \in \nat$ and $z \in E$ and write $\hat{Z}_p = \argmin_{z} \| \bX_{2:p} - \boldsymbol{\mu}_{2:p}(z)\|_2$ 
for the prediction of $Z$ based on the coordinates $2, \ldots, p$.
With this oracle choice of $\check{\boldsymbol{\mu}}_{1:p}(z)$, the 
relative errors are zero, thus the bound \eqref{eq:chebmis} holds, and Lemma \ref{lem:chebyshev} 
gives 
\begin{align*}
    \P\left(\hat{Z}_p \neq Z\right) & = \sum_{z} \sum_{v \neq z} 
    \P\left(\hat{Z}_p = v, Z = z\right) \\
    & =  \sum_{z} \sum_{v \neq z} \P\left(\hat{Z}_p = v \given Z = z\right) \P\left(Z = z\right) \\
    & \leq \frac{C}{\min_{z \neq v} \norm{\boldsymbol{\mu}_{2:p}(z) - \boldsymbol{\mu}_{2:p}(v)}_2^2}
\end{align*}
with $C$ a constant independent of $p$. By \eqref{eq:sep}, $\min_{z \neq v} \norm{\boldsymbol{\mu}_{2:p}(z) - \boldsymbol{\mu}_{2:p}(v)}_2^2 \to \infty$ for $p \to \infty$, and by choosing a subsequence, $p_r$,
we can ensure that  
$\P\left(\hat{Z}_{p_r} \neq Z\right) \leq \frac{1}{r^2}.$ Then $\sum_{r=1}^{\infty} \P\left(\hat{Z}_{p_r} \neq Z\right) < \infty$, and by Borel-Cantelli's lemma, 
\[
\P\left(\hat{Z}_{p_r} \neq Z \ \text{infinitely often}\right) = 0.
\]
That is, $\P\left(\hat{Z}_{p_r} = Z \ \text{eventually}\right) = 1$, which shows 
that we can recover $Z$ from $(\hat{Z}_{p_r})_{r \in \nat}$ and thus from $\bX_{-1}$ 
(with probability $1$). Defining 
\[
Z' = \left\{ \begin{array}{ll} \lim\limits_{r \to \infty} \hat{Z}_{p_r} & \qquad 
\text{if } \hat{Z}_{p_r} = Z \text{ eventually} \\
0 & \qquad \text{otherwise} \end{array} \right.
\]
we see that $\sigma(Z') \subseteq \sigma(\bX_{-1})$ and $Z' = Z$ almost surely. Thus 
if we replace $Z$ by $Z'$ in Assumption \ref{ass:mixture} we see that Assumption \ref{ass:variable}(2) 
holds. 
\end{proof}

\begin{lemma} \label{lem:subgauss}
Consider the same setup as in Lemma \ref{lem:chebyshev}, that is, 
Assumption \ref{ass:mixture} holds and $R_{z,v}^{(p)} \leq \frac{1}{10}$ for all $z, v \in E$ with $v \neq z$. Suppose, in addition, that the conditional distribution of 
$X_i$ given $Z = z$ is sub-Gaussian with variance factor $v_{\max}$, independent of 
$i$ and $z$, then 
\begin{equation} \label{eq:suggaussmis}
\P(\hat Z=v \mid Z=z) \leq \exp\left( - \frac{1}{50 v_{\max}}
\| \boldsymbol{\mu}_{1:p}(z) - \boldsymbol{\mu}_{1:p}(v)\|_2^2 \right).
\end{equation}
\end{lemma}

\begin{proof} Recall that $X_i$ given $Z = z$ being sub-Gaussian with variance factor $v_{\max}$
means that 
\[
\log \ex\left[ e^{\lambda(X_i - \mu_i(z))} \given Z = z \right] \leq \frac{1}{2}\lambda^2v_{\max}
\]
for $\lambda \in \real$. Consequently, with $W_i$ as in the proof of Lemma \ref{lem:chebyshev}, and using conditional independence of the $X_i$-s given $Z = z$, 
\begin{align*}
\log \ex\left[ e^{\lambda \sum_{i=1}^p W_i} \given Z = z \right] & =
\sum_{i=1}^p \log \ex\left[ e^{\lambda(\check{\mu}_i(z) - \check{\mu}_i(v))(X_i - \mu_i(z))} \given Z = z \right] \\
& \leq \frac{1}{2}\lambda^2v_{\max}\sum_{i=1}^p (\check{\mu}_i(z) - \check{\mu}_i(v))^2 \\
& = \frac{1}{2} \lambda^2v_{\max}\| \check{\boldsymbol{\mu}}_{1:p}(z) - \check{\boldsymbol{\mu}}_{1:p}(v)\|_2^2.
\end{align*}    
Using \eqref{eq:Zinclusion} in combination with the Chernoff bound gives
\begin{align*}
    \P(\hat Z = v \mid Z = z) & \leq \P\left(
    \sum_{i=1}^p W_i < - \tfrac{6}{25} \norm{\boldsymbol{\mu}_{1:p}(z) - \boldsymbol{\mu}_{1:p}(v)}_2^2 
    \given  Z = z \right) \\
    & \leq \exp\left( - \left(\frac{6}{25}\right)^2 \frac{\norm{\boldsymbol{\mu}_{1:p}(z) - \boldsymbol{\mu}_{1:p}(v)}_2^4} 
    {2 v_{\max}\| \check{\boldsymbol{\mu}}_{1:p}(z) - \check{\boldsymbol{\mu}}_{1:p}(v)\|_2^2} \right) \\
    & =  \exp\left( - \frac{1}{2 v_{\max}} \left(\frac{6}{25}\right)^2 B_{z,v}^{-2} \norm{\boldsymbol{\mu}_{1:p}(z) - \boldsymbol{\mu}_{1:p}(v)}_2^2 \right) \\
    & \leq \exp\left( - \frac{1}{50 v_{\max}}  \norm{\boldsymbol{\mu}_{1:p}(z) - \boldsymbol{\mu}_{1:p}(v)}_2^2 \right),
\end{align*}
where we, as in the proof of Lemma \ref{lem:chebyshev}, have used that the bound on the relative error implies that $B_{z,v} \leq \tfrac{6}{5}$.
\end{proof}

\begin{proof}[Proof of Proposition \ref{prop:error_rate_bound}] The argument proceeds as in 
the proof of Proposition \ref{prop:recover}. We first note that 
\begin{align*}
    \P\left(\hat{Z} \neq Z\right) & = \sum_{z} \sum_{v \neq z} 
    \P\left(\hat{Z} = v, Z = z\right) \\
    & =  \sum_{z} \sum_{v \neq z} \P\left(\hat{Z} = v \given Z = z\right) \P\left(Z = z\right). 
\end{align*}
Lemma \ref{lem:chebyshev} then gives 
\[
\P\left(\hat{Z} \neq Z\right) \leq \frac{25 K \sigma^2_{\max}}{\sep(p)}.
\]
If the sub-Gaussian assumption holds, Lemma \ref{lem:subgauss} instead gives 
\[
\P\left(\hat{Z} \neq Z\right) \leq K \exp\left( - \frac{\sep(p) }{50 v_{\max}}
\right).
\]
\end{proof}

\subsection{Proof of Theorem \ref{thm:main}} \label{app:main_proof}

\begin{proof}[Proof of Theorem \ref{thm:main}] Recall that 
\[
\delta = \frac{1}{n} \sum_{k=1}^n 1(\hat{z}_k \neq z_k),
\]
hence by Proposition \ref{prop:error_rate_bound}
\begin{align}
    \ex[\delta] & = \P(\hat{Z}_k \neq Z) \nonumber \\
    & \leq \P\left(\hat{Z}_k \neq Z \given \max_{z \neq v} R^{(p)}_{z,v} \leq \tfrac{1}{10} \right) 
    + \P\left(\max_{z \neq v} R^{(p)}_{z,v} > \tfrac{1}{10}  \right) \nonumber \\
    & \leq \frac{25 K \sigma_{\max}^2}{\sep(p)} + K^2  \max_{z \neq v} \P\left(R^{(p)}_{z,v} > \tfrac{1}{10}  \right). \label{eq:delta_mean_bound}
\end{align}
Both of the terms above tend to $0$, thus $\delta \overset{P}{\to} 0$.

Now rewrite the bound \eqref{eq:bias_bound} as 
\[
|\widehat{\beta}^{\s}_i - \hat{\beta}_i| \leq \sqrt{\delta} \underbrace{\left(\frac{2\sqrt{2}}{\rho^2 \sqrt{\alpha}} \frac{\|\by\|_2}{\|\bx_i\|_2} \right)}_{= L_n}
\]
From the argument above, $\sqrt{\delta} \overset{P}{\to} 0$. We will show 
that the second factor, $L_n$, tends to a constant, $L$, 
in probability under the stated assumptions. This will imply that 
\[ 
    |\widehat{\beta}^{\s}_i - \hat{\beta}_i| \overset{P}{\to} 0,
\]
which shows case (1). 

Observe first that 
\[
\|\bx_i\|_2^2 = \frac{1}{n} \sum_{k=1}^n x_{i,k}^2 \overset{P}{\to}
\ex[X_i^2] \in (0, \infty)
\]
by the Law of Large Numbers, using the i.i.d. assumption 
and the fact that $\ex[X_i^2] \in (0,\infty)$ by Assumption \ref{ass:mixture}. Similarly, 
$\|\by\|_2^2 \overset{P}{\to} \ex[Y] \in [0, \infty)$.

Turning to $\alpha$, we first see that by the Law of Large Numbers,
\[
\frac{n(z)}{n} \overset{P}{\to} \P(Z = z)
\]
for $n \to \infty$ and $z \in E$. Then observe that 
for any $z \in E$
\[
| \hat{n}(z)  - n(z) | \leq \sum_{k=1}^n |1(\hat{z}_k = z) - 1(z_k = z)| \leq  
\sum_{k=1}^n 1(\hat{z}_k \neq z_k) \leq n \delta. 
\]
Since $\delta \overset{P}{\to} 0$, also 
\[
\frac{\hat{n}(z)}{n} \overset{P}{\to} \P(Z = z),
\]
thus 
\[
\alpha = \frac{n_{\min}}{n} = \min \left\{ \frac{n(1)}{n}, \ldots, \frac{n(K)}{n},
\frac{\hat{n}(1)}{n}, \ldots, \frac{\hat{n}(K)}{n} \right\}
\overset{P}{\to} \min_{z \in E} \P(Z = z) \in (0, \infty).
\]

We finally consider $\rho$, and to this end we first see that 
\begin{align*}
\frac{1}{n} \| (I - P_{\bZ}) \bx_{i} \|_2^2 = 
\frac{1}{n} \sum_{k=1}^n (x_{i,k} - \overline{\mu}(z_k))^2 \overset{P}{\to} 
\ex\left[\sigma_i^2(Z) \right] \in (0, \infty).            
\end{align*}
Moreover, using Lemma \ref{lem:proj_norm}, 
\begin{align*}
\left| \| (I - P_{\hat{\bZ}}) \bx_i \|_2^2 - \| (I - P_{\bZ}) \bx_i \|_2^2 \right| 
& = \left|\| (P_{\hat{\bZ}} - P_{\bZ}) \bx_i \|_2^2 + 
2 |\langle (I - P_{\hat{\bZ}}) \bx_i , (P_{\hat{\bZ}} - P_{\bZ}) \bx_i \rangle \right| \\
& \leq \| P_{\hat{\bZ}} - P_{\bZ} \|_2^2\| \bx_i \|_2^2 + 
2 \| P_{\hat{\bZ}} - P_{\bZ} \|_2 \| \bx_i \|^2_2 \\
& \leq \left( \frac{2 \delta}{ \alpha}+ \sqrt{\frac{2\delta}{\alpha}} \right) \| \bx_i \|^2_2. 
\end{align*}
Hence 
\[
\rho \overset{P}{\to}  \frac{\ex\left[\sigma_i^2(Z) \right]}{\ex[X_i^2]} \in (0,\infty). 
\]

Combining the limit results, 
\[
L_n \overset{P}{\to}  L = \frac{2\sqrt{2}\ex[X_i^2]^2}{\ex\left[\sigma_i^2(Z)\right]^2 \sqrt{\min_{z \in E} \P(Z = z)}} 
\sqrt{\frac{\ex[Y^2]}{\ex[X_i^2]}} \in (0, \infty).
\]

To complete the proof, suppose first that $\frac{\sep(p)}{n} \to \infty$.
Then 
\[
\sqrt{n} |\widehat{\beta}^{\s}_i - \hat{\beta}_i| \leq \sqrt{n \delta} L_n 
\]
By \eqref{eq:delta_mean_bound} we have, under the assumptions given in case (2) 
of the theorem, that $n \delta \overset{P}{\to} 0$, and case (2) follows. 

Finally, in the sub-Gaussian case, and if just $h_n = \frac{\sep(p)}{\log(n)} \to \infty$, 
then we can replace \eqref{eq:delta_mean_bound} by the bound 
\[
    \ex[\delta] \leq K \exp\left( - \frac{\sep(p)}{50 v_{\max}} \right) + K^2  \max_{z \neq v} \P\left(R^{(p)}_{z,v} > \tfrac{1}{10}  \right).
\]
Multiplying by $n$, we get that the first term in the bound equals
\begin{align*}
K n \exp\left( - \frac{\sep(p)}{50 v_{\max}} \right)  
& = K \exp\left( - \frac{\sep(p)}{50 v_{\max}} + \log(n) \right)   \\
& = K \exp\left( \log (n) \left(1 - \frac{h_n}{50 v_{\max}} \right) \right)  \to 0
\end{align*}
for $n \to \infty$. We conclude that the relaxed 
growth condition on $p$ in terms of $n$ in the sub-Gaussian case is enough to 
imply $n \delta \overset{P}{\to} 0$, and case (3) follows. 

By the decomposition 
\[
\sqrt{n}(\widehat{\beta}^{\s}_i - \beta_i) = \sqrt{n} (\widehat{\beta}^{\s}_i - \widehat \beta_i) + \sqrt{n}(\hat\beta_i - \beta_i)
\]
it follows from Slutsky's theorem that in case (2) as well as case (3),
\[
\sqrt{n} (\widehat{\beta}^{\s}_i - \beta_i)  = \sqrt{n}(\widehat{\beta}_i - \beta_i) + o_P(1) \overset{\mathcal{D}}{\to} \mathcal{N}(0, w_i^2).
\]
\end{proof}

\section{Gaussian mixture models} \label{app:kakutani}
This appendix contains an analysis of a latent variable model with a 
finite $E$, similar to the one given by Assumption \ref{ass:mixture}, 
but with Assumption \ref{ass:mixture}(1) strengthened to
\[
X_i \mid Z = z \sim \mathcal{N}(\mu_i(z), \sigma_i^2(z)).
\]
Assumptions \ref{ass:mixture}(2), \ref{ass:mixture}(3) and \ref{ass:mixture}(4) are 
dropped, and the purpose is to understand precisely when 
Assumption \ref{ass:variable}(2) holds in this model. That is, when 
$Z$ can be recovered from $\bX_{-i}$. To keep notation simple, 
we will show when $Z$ can be recovered from $\bX$, but the 
analysis and conclusion is the same if we left out a single coordinate. 

The key to this analysis is a classical result due to Kakutani. As 
in Section \ref{sec:setup}, the conditional distribution of $\bX$ 
given $Z = z$ is denoted $P_z$, and the model assumption is that 
\begin{equation} \label{eq:product_gauss}
P_z = \bigotimes_{i=1}^\infty P^i_z
\end{equation}
where $P^i_z$ is the conditional distribution of $X_i$ given $Z = z$. For 
Kakutani's theorem below we do not need the Gaussian assumption; only that 
$P^i_z$ and $P^i_v$ are equivalent (absolutely continuous w.r.t. each other), and 
we let $\frac{\mathrm{d} P_z^i}{\mathrm{d} P_v^i}$ denote the Radon-Nikodym 
derivative of $P_z^i$ w.r.t. $P_v^i$.

\begin{theorem}[\cite{Kakutani}] \label{thm:kakutani}
Let $z, v \in E$ and $v \neq z$. Then $P_z$ 
and $P_v$ are singular if and only if 
\begin{equation}
    \sum_{i=1}^\infty - \log \int \sqrt{\frac{\mathrm{d} P_z^i}{\mathrm{d} P_v^i}}\ \mathrm{d} P_v^i = \infty.
\end{equation}
\end{theorem}

Note that 
\[
\mathrm{BC}_{z,v}^i = \int \sqrt{\frac{\mathrm{d} P_z^i}{\mathrm{d} P_v^i}}\ \mathrm{d} P_v^i 
\]
is known as the Bhattacharyya coefficient, while $- \log(\mathrm{BC}_{z,v}^i)$ 
and $\sqrt{1 - \mathrm{BC}_{z,v}^i}$ are known as the Bhattacharyya distance and the 
Hellinger distance, respectively, between $P^i_z$ and $P^i_v$. Note also 
that if $P^i_ z = h^i_z \cdot \lambda$ and $P^i_v = h^i_v \cdot \lambda$ for a reference 
measure $\lambda$, then 
\[
\mathrm{BC}_{z,v}^i = \int \sqrt{h_z^i h_v^i} \ \mathrm{d} \lambda.
\]

\begin{proposition} \label{prop:kakutani}
Let $P^i_z$ be the $\mathcal{N}(\mu_i(z), \sigma_i^2(z))$-distribution 
for all $i \in \nat$ and $z \in E$. Then $P_z$ 
and $P_v$ are singular if and only if either 
\begin{align} \label{eq:kakutani1}
    \sum_{i=1}^{\infty}  \frac{(\mu_{i}(z) - \mu_{i}(v))^2}{\sigma_{i}^2(z) + \sigma^2_{i}(v)} & = \infty
    \qquad \text{or} \\ \label{eq:kakutani2}
\sum_{i=1}^{\infty} \log\left(\frac{\sigma_{i}^2(z) + \sigma^2_{i}(v)}{2\sigma_{i}(z)\sigma_{i}(v)}\right)
& = \infty
\end{align}
    
\end{proposition}

\begin{proof} Letting $\mu = \mu_i(z)$, $\nu = \mu_i(v)$, $\tau = 1/\sigma_i(z)$
and $\kappa = 1/\sigma_i(v)$ we find 
    
\begin{align*}
    \mathrm{BC}_{z,v}^i
    & =  \int \sqrt{
        \frac{\tau}{\sqrt{2\pi}}
        \exp(-\frac{\tau^2}{2}\qty(x-\mu)^2)
        \frac{\kappa}{\sqrt{2\pi}}
        \exp(-\frac{\kappa^2}{2}\qty(x-\nu)^2)
    }  \mathrm{d} x\\
    &= \sqrt{\frac{\tau\kappa}{2\pi}} \int \exp(
        -\frac{
            (\tau^2 + \kappa^2) x^2 
            - 2(\tau^2\mu + \kappa^2\nu) x 
            + (\tau^2\mu^2 + \kappa^2\nu^2)
        }{4}
    ) \mathrm{d} x \\
    &= \sqrt{\frac{\tau\kappa}{2\pi}} 
    \sqrt{\frac{4\pi}{\tau^2 + \kappa^2}}
    \exp(
        \frac{(\tau^2\mu + \kappa^2\nu)^2}
        {4(\tau^2 + \kappa^2)} 
        - \frac{\tau^2\mu^2 + \kappa^2\nu^2}{4}
    )\\
    &= \sqrt{\frac{2\tau\kappa}{\tau^2 + \kappa^2}}
    \exp(
        -\frac{\tau^2\kappa^2(\mu-\nu)^2}
        {4(\tau^2 + \kappa^2)} 
    ) \\
    & = \sqrt{\frac{2\sigma_i(z) \sigma_i(v)}{\sigma^2_i(z) + \sigma^2_i(z)}}
    \exp(
        -\frac{(\mu_i(z) - \mu_i(v))^2}
        {4(\sigma^2_i(z) + \sigma^2_i(z))} 
    ).
\end{align*}
Thus 
\[
\sum_{i=1}^{\infty} - \log \left(\mathrm{BC}_{z,v}^i \right) = 
  \frac{1}{2} \sum_{i=1}^\infty \log \left(\frac{\sigma^2_{i}(z) + \sigma_{i}^2(v)}{2\sigma_{i}(z)\sigma_{i}(v)}\right) + 
  \frac{1}{4}\sum_{i=1}^{\infty}   \frac{(\mu_{i}(z) - \mu_{i}(v))^2}{\sigma^2_{i}(z) + \sigma_{i}^2(v)},
\]
and the result follows from Theorem \ref{thm:kakutani}.
\end{proof}

\begin{corollary} \label{cor:kakutani}
Let $P^i_z$ be the $\mathcal{N}(\mu_i(z), \sigma_i^2(z))$-distribution 
for all $i \in \nat$ and $z \in E$. There is a mapping $f : \real^\nat \to E$ 
such that $Z = f(\bX)$ almost surely if and only if either \eqref{eq:kakutani1}
or \eqref{eq:kakutani2} holds. 
\end{corollary}

\begin{proof} If  either \eqref{eq:kakutani1} or \eqref{eq:kakutani2} holds, 
$P_z$ and $P_v$ are singular whenever $v \neq z$. This implies that there are 
measurable subsets $A_z \subseteq \real^\nat$ for $z \in E$ such that $P_z(A_z) = 1$
and $P_v(A_z) = 0$ for $v \neq z$. Setting $A = \cup_z A_z$ we see that 
\[
P(A) = \sum_{z} P_z(A) \P(Z = z) = \sum_{z} P_z(A_z) \P(Z = z) = 1.
\]
Defining the map $f : \real^\nat \to E$ by $f(\bx) = z$ if $\bx \in A_z$ (and arbitrarily on the complement of $A$) we see that $f(\bX) = Z$ almost surely. 

On the other hand, if there is such a mapping $f$, define $A_z = f^{-1}(\{z\})$ for all $z \in E$.
Then $A_z \cap A_v = \emptyset$ for $v \neq z$ and 
\begin{align*}
P_z(A_z) & = \frac{\P(\bX \in A_z, Z = z)}{\P(Z = z)} = \frac{\P(f(\bX) = z, Z = z)}{\P(Z = z)} \\
& = \frac{\P(f(\bX) = Z, Z = z)}{\P(Z = z)} = \frac{\P(Z = z)}{\P(Z = z)} = 1.
\end{align*}
Similarly,  for $v \neq z$
\begin{align*}
P_v(A_z) & = \frac{\P(\bX \in A_z, Z = v)}{\P(Z = v)} = \frac{\P(f(\bX) = z, Z = v)}{\P(Z = v)} \\
& = \frac{\P(f(\bX) \neq Z, Z = v)}{\P(Z = v)} = \frac{0}{\P(Z = v)} = 0.
\end{align*}
This shows that $P_z$ and $P_v$ are singular, and by Proposition \ref{prop:kakutani}, 
either \eqref{eq:kakutani1} or \eqref{eq:kakutani2} holds.
\end{proof}

\bibliographystyle{agsm}
\bibliography{bibliography}

\end{document}